\documentclass[12pt,english]{article}
\usepackage[T1]{fontenc}
\usepackage[latin9]{inputenc}
\usepackage{geometry}
\usepackage{amsmath}
\usepackage{amsthm}
\usepackage{amssymb}
\usepackage{stackrel}

\makeatletter
\numberwithin{equation}{section}
\theoremstyle{plain}
\newtheorem{thm}{\protect\theoremname}[section]
\theoremstyle{plain}
\newtheorem{prop}[thm]{\protect\propositionname}
\theoremstyle{definition}
\newtheorem{defn}[thm]{\protect\definitionname}
\theoremstyle{plain}
\newtheorem{lem}[thm]{\protect\lemmaname}
\theoremstyle{remark}
\newtheorem{rem}[thm]{\protect\remarkname}
\theoremstyle{definition}
\newtheorem{example}[thm]{\protect\examplename}
\theoremstyle{plain}
\newtheorem{cor}[thm]{\protect\corollaryname}

\usepackage{graphicx}
\usepackage{float}
\usepackage{hyperref}
\hypersetup{
     colorlinks   = true,
     citecolor    = blue,
     linkcolor    = blue
}

\date{}

\makeatletter
\newsavebox{\@brx}
\newcommand{\llangle}[1][]{\savebox{\@brx}{\(\m@th{#1\langle}\)}%
  \mathopen{\copy\@brx\kern-0.5\wd\@brx\usebox{\@brx}}}
\newcommand{\rrangle}[1][]{\savebox{\@brx}{\(\m@th{#1\rangle}\)}%
  \mathclose{\copy\@brx\kern-0.5\wd\@brx\usebox{\@brx}}}
\makeatother

\makeatother

\usepackage{babel}
\providecommand{\corollaryname}{Corollary}
\providecommand{\definitionname}{Definition}
\providecommand{\examplename}{Example}
\providecommand{\lemmaname}{Lemma}
\providecommand{\propositionname}{Proposition}
\providecommand{\remarkname}{Remark}
\providecommand{\theoremname}{Theorem}

\begin{document}
\title{Non-degeneracy of Stochastic Line Integrals}
\author{Xi Geng\thanks{School of Mathematics and Statistics, University of Melbourne, Parkville
VIC 3010, Australia. Email: xi.geng@unimelb.edu.au. XG acknowledges
the support from ARC Grant DE210101352.}$\ $ and Sheng Wang\thanks{School of Mathematics and Statistics, University of Melbourne, Parkville
VIC 3010, Australia. Email: shewang4@student.unimelb.edu.au.}}
\maketitle
\begin{abstract}
We derive quantitative criteria for the existence of density for stochastic
line integrals and iterated line integrals along solutions of hypoelliptic
differential equations driven by fractional Brownian motion. As an
application, we also study the signature uniqueness problem for these
rough differential equations.
\end{abstract}

\section{Introduction and summary of main results}

It is classical that there is a natural pairing between a ${\cal C}^{1}$-path
$\gamma:[0,T]\rightarrow M$ in a differentiable manifold $M$ and
a differential one-form $\phi$ on $M$, which is defined by integration:
\[
\int_{0}^{T}\phi(d\gamma_{t})\triangleq\int_{0}^{T}\langle\phi,\dot{\gamma}_{t}\rangle dt.
\]
Here $\langle\cdot,\cdot\rangle$ denotes the pairing between cotangent
and tangent vectors. This notion of integration, sometimes known as
\textit{line integrals}, has an intrinsic geometric meaning in the
sense that it does not rely on local coordinates or embeddings of
$M$ into ambient Euclidean spaces. More generally, given a finite
sequence of one-forms $(\phi_{1},\cdots,\phi_{m})$, one can consider
an associated iterated line integral 
\[
\int_{0<t_{1}<\cdots<t_{m}<T}\phi_{1}(d\gamma_{t_{1}})\cdots\phi_{m}(d\gamma_{t_{m}})\triangleq\int_{0}^{T}\int_{0}^{t_{m}}\cdots\int_{0}^{t_{2}}\langle\phi_{1},\dot{\gamma}_{t_{1}}\rangle\cdots\langle\phi_{m},\dot{\gamma}_{t_{m}}\rangle dt_{1}\cdots dt_{m}.
\]
The definition of such integrals can be naturally extended to the
rough path context under suitable regularity conditions on the path
$\gamma$ and the one-forms (cf. \cite{LQ02,CDL15}). In the rough
path literature, these iterated line integrals are often referred
to as \textit{extended signatures} of $\gamma$ (cf. \cite{LQ12}
for their use in the context of Brownian motion).

A natural reason of considering line integrals is that they encode
rich geometric/topological information about the original path $\gamma.$
For instance, if $\gamma=(x_{t},y_{t})$ is a simple closed curve
in $\mathbb{R}^{2}$, the line integral of $\gamma$ against the one-form
\begin{equation}
\phi\triangleq\frac{1}{2}(xdy-ydx)\label{eq:AreaForm}
\end{equation}
gives the (signed) area enclosed by the path $\gamma$. The integral
against the one-form 
\[
d\theta\triangleq\frac{1}{r^{2}}(xdy-ydx)
\]
on the punctured plane gives the winding number of $\gamma$ around
the origin. Other topological properties associated with paths, e.g.
turning number and linking number, can also be defined in a similar
way in terms of line integrals. In the probabilistic context, one
can considder distributional properties of stochastic line integrals
along stochastic processes such as diffusion paths. A well-known example
is L\'evy's formula for the characteristic function of the area process
associated with a planar Brownian motion, i.e. the stochastic line
integral of Brownian motion against the area one-form defined by (\ref{eq:AreaForm})
(cf. \cite{Lev40}). Another famous example is Spitzer's asymptotic
Cauchy law for the Brownian winding number (cf. \cite{Spi58}). Stochastic
line integrals are also essential in the study of diffusions/martingales
on manifolds (cf. \cite{Hsu02}).

A more fundamental reason of considering (iterated) line integrals
is that the original path $\gamma$ is uniquely determined by these
integrals when one varies the degree $n$ and the one-forms $\phi_{1},\cdots,\phi_{n}$
in a suitably rich class. Indeed, when $M=\mathbb{R}^{d}$, the collection
of numbers (known as the \textit{signature} of $\gamma$)
\[
\big\{\int_{0<t_{1}<\cdots<t_{m}<T}d\gamma_{t_{1}}^{i_{1}}\cdots d\gamma_{t_{m}}^{i_{m}}:m\in\mathbb{N},\ i_{1},\cdots,i_{m}=1,\cdots,d\big\}
\]
uniquely determines the path $\gamma:[0,T]\rightarrow\mathbb{R}^{d}$
up to tree-like pieces (cf. \cite{Che58,HL10,BGLY16}). In \cite{Che73},
the author used iterated line integrals against differential forms
to construct a de Rham cohomology theory on loop spaces over manifolds
and proved that such a theory is canonically isomorphic to the singular
cohomology theory in classical algebraic topology.

In the probabilistic context, in the pioneering work of Le Jan and
Qian \cite{LQ12}, the authors developed an explicit method of recovering
a generic Brownian trajectory from the knowledge of its extended signatures.
Their underlying idea can be summarised as follows. Given an arbitrary
bounded domain $D$ in $\mathbb{R}^{d}$, by constructing a suitable
one-form $\phi$ supported on $D$ one can detect whether the Brownian
motion $B$ has visited $D$ from the knowledge of the line integral
against $\phi.$ More generally, given a discretisation of $\mathbb{R}^{d}$
into disjoint cubes with suitably constructed one-forms supported
inside each of them, one can detect the discrete route of the motion
from the knowledge of iterated line integrals against these on-forms.
By refining the space discretisation, one recovers the original trajectory
in the limit under this mechanism (cf. Section \ref{sec:UoS} below
for more discussion).

In the method of \cite{LQ12}, an essential property of the required
one-form $\phi$ is that 
\[
\int_{0}^{T}\phi(dB_{t})\neq0{\rm \iff}B\text{ visits the }D\ \ \ \text{a.s}.
\]
where $B$ is a Brownian motion in $\mathbb{R}^{d}$. Such a property
can be trivially implied by a much stronger non-degeneracy property
that the conditional law of $\int_{0}^{T}\phi(dB_{t})$ given that
$B$ visits the domain $D$ is absolutely continuous with respect
to the Lebesgue measure. This motivates the following general question
which is the main object of study in the present work.

\vspace{2mm} We consider the following SDE on $M$ ($M=\mathbb{R}^{n}$
or a compact differentiable manifold):

\begin{equation}
\begin{cases}
dX_{t}=\sum_{\alpha=1}^{d}V_{\alpha}(X_{t})dB_{t}^{\alpha}, & 0\leqslant t\leqslant T;\\
X_{0}=x_{0}\in M.
\end{cases}\label{eq:MainRDE}
\end{equation}
Here $B=(B^{1},\cdots,B^{d})$ is assumed to be a $d$-dimensional
fractional Brownian motion with Hurst parameter $H>1/4$. This falls
into the rough path framework under which the SDE is well-posed in
the sense of rough paths. The vector fields $V_{1},\cdots,V_{d}$
on $M$ are assumed to be of class $C_{b}^{\infty}$ and satisfy\textit{
}the so-called \textit{H\"ormander's condition} (cf. Definition \ref{def:HorCond}).
This is a natural non-degeneracy condition under which the solution
$X_{t}$ is known to have a smooth density function with respect to
the Lebesgue measure (cf. \cite{CHLT15}). Throughout the rest, we
use $C_{p}^{\infty}$ to mean the class of functions/one-forms whose
derivatives (of all orders) have at most polynomial growth. This property
ensures the $L^{p}$-integrability (for all $p>1$) of all relevant
random variables under consideration.

\vspace{2mm}\noindent \textit{Question}. Let $\phi$ be a $C_{p}^{\infty}$
one-form on $M$. Can we identify an explicit quantitative condition
on $\phi$, such that the conditional distribution of the stochastic
line integral $\int_{0}^{T}\phi(dX_{t})$, given that $X$ visits
the interior of the support of $\phi$, admits a density function
with respect to the Lebesgue measure?

\vspace{2mm} We first make a few comments. It is necessary to restrict
on the event that $X$ visits $({\rm supp}\phi)^{\circ}$, for otherwise
the line integral is trivially zero. In addition, suppose that $M=\mathbb{R}^{n}$,
${\rm supp}\phi\neq M$ and $x_{0}\in({\rm supp\phi})^{\circ}$. For
the stochastic line integral $\int_{0}^{T}\phi(dX_{t})$ to have a
density function, it is necessary that $\phi$ is not closed. Indeed,
if $d\phi=0$, then $\phi=df$ for some smooth function $f$ (every
closed one-form on $\mathbb{R}^{n}$ is exact). In this case, we have
\[
\int_{0}^{T}\phi(dX_{t})=f(X_{T})-f(x_{0}).
\]
This integral will have constant value on the non-trivial event $\{X_{T}\notin{\rm supp\phi}\}.$
As a result, the line integral cannot have a density function in this
case.

As we will see, in the elliptic case, the non-closedness of $\phi$
is essentially sufficient for the line integral to have a density.
\begin{thm}
Suppose that the vector fields $V_{1},\cdots,V_{d}$ are elliptic.
Let $\phi$ be a $C_{p}^{\infty}$ one-form such that 
\[
d\phi\neq0\ \ \ \text{a.e. }\text{on }{\rm supp\phi.}
\]
Then the conditional distribution of $\int_{0}^{T}\phi(dX_{t})$,
given that $X$ visits $({\rm supp\phi})^{\circ}$, admits a density
with respect to the Lebesgue measure.
\end{thm}

The hypoelliptic case requires a stronger condition and more delicate
analysis. The general result is given by Theorem \ref{thm:GenMain}
below. Here we state the special version in the step-two hypoelliptic
case.
\begin{thm}
\label{thm:Step2Intro}Consider the case when $M=\mathbb{R}^{3}$
and $d=2$. Suppose that the vector fields ${\cal V}=\{V_{1},V_{2},[V_{1},V_{2}]\}$
linear span $T_{x}M$ at every point $x\in M$. Let $\phi$ be a $C_{p}^{\infty}$
one-form on $M$. Suppose that 
\[
d\big(\phi+d\phi(V_{1},V_{2})\omega^{3}\big)\neq0\ \ \ \text{a.e.\ on }{\rm supp\phi},
\]
where $\{\omega^{1},\omega^{2},\omega^{3}\}$ is the cotangent frame
field dual to ${\cal V}$. Then the conditional distribution of $\int_{0}^{T}\phi(dX_{t})$,
given that $X$ visits $({\rm supp\phi})^{\circ}$, admits a density
with respect to the Lebesgue measure.
\end{thm}

In Theorem \ref{thm:NonDegHypo} below, we also derive an explicit
method of constructing one-forms that satisfy the general non-degeneracy
criterion given by Theorem \ref{thm:GenMain}. In the above step-two
hypoelliptic case, the method is summarised in the following result.
In this case, the class of one-forms that satisfy such a condition
is as generic as pairs of $C_{p}^{\infty}$-functions.
\begin{prop}
Under the setting of Theorem \ref{thm:Step2Intro}, consider a one-form
$\phi$ given by 
\[
\phi\triangleq c_{1}\omega^{1}+c_{2}\omega^{2}+(V_{1}c_{2}-V_{2}c_{1})\omega^{3}
\]
where $c_{1},c_{2}\in C_{p}^{\infty}(M)$. Suppose that 
\begin{equation}
d\phi\neq0\ \ \ \text{a.e. }\text{on }{\rm supp\phi.}\label{eq:S2ConsIntro}
\end{equation}
Then the conditional distribution of $\int_{0}^{T}\phi(dX_{t})$,
given that $X$ visits $({\rm supp\phi})^{\circ}$, admits a density
with respect to the Lebesgue measure.
\end{prop}

Our analysis can be extended to the case of iterated line integrals
\[
F\triangleq\int_{0<t_{1}<\cdots<t_{m}<T}\phi_{1}(dX_{t_{1}})\cdots\phi_{m}(dX_{t_{m}}).
\]
As we will see, if $\phi_{1},\cdots,\phi_{m}$ have disjoint supports,
our general condition given by (\ref{eq:GenMain}) imposed on each
$\phi_{i}$ continues to guarantee the conditional non-degeneracy
of $F.$ On the other hand, if these one-forms have a common compact
support, when $m\geqslant2$ it is indeed possible to have all the
$\phi_{i}$'s being exact while $F$ is non-degenerate. Recall what
we explained earlier that this is not possible when $m=1$. Our results
for iterated line integrals are discussed in Section \ref{subsec:ILI}
below.

Our study is motivated by the signature uniqueness problem in the
spirit of \cite{LQ12}. As an application, in Section \ref{sec:UoS}
we prove a signature uniqueness theorem for the SDE (\ref{eq:MainRDE})
in the elliptic or step-two hypoelliptic case, which asserts that
with probability one the solution path $t\mapsto X_{t}$ is uniquely
determined by its signature transform up to reparametrisation. Under
existing methodology, the key ingredient is the explicit construction
of compactly supported one-forms that satisfy our non-degeneracy conditions.
The main result for this part is stated in Theorem \ref{thm:ASSigUniq}
below.

Finally, we remark that our results hold for more general Gaussian
driving processes essentially without changing any part of the proofs.
The required Gaussian setting is precisely the one formulated in the
work of \cite{CHLT15} concerning the smoothness of density for Gaussian
rough differential equations. We formulate our results in the context
of fractional Brownian motion simply to avoid the non-rewarding effort
of restating all the assumptions proposed in \cite{CDL15}.

\vspace{2mm}\noindent \textbf{Organisation}. The present article
is organised in the following way. In Section 2, we recall basic notions
from rough path theory and some terminology from differential geometry.
In Section \ref{subsec:SingleLI}, we derive our quantitative criteria
for the non-degeneracy of single stochastic line integrals as well
as an explicit method of construction. We begin with the elliptic
case and then proceed to the hypoelliptic case. The analysis is made
more transparent in the step-two hypoelliptic case after the general
discussion. In Section \ref{subsec:ILI}, we extend our analysis to
the case of iterated line integrals. In Section \ref{sec:UoS}, we
discuss the application of our results to the signature uniqueness
problem for rough differential equations.

\section{Preliminary notions from rough path theory and differential geometry}

In this section, we recall some basic tools and discuss the basic
kind of pathwise analysis that will be performed frequently in the
sequel. We first give a notational comment which will be applied throughout
the rest of the article.

\vspace{2mm}\noindent \textbf{Notation}. Above all, we will adopt
Einstein's convention of summation, i.e. doubly repeated indices are
summed automatically. We will also use matrix notation exclusively.
For instance, a vector field $V=V^{i}\frac{\partial}{\partial x^{i}}$
on $\mathbb{R}^{n}$ is identified as an $n\times1$ column vector
function. $DV$ is the $n\times n$ matrix whose $(i,j)$-entry is
$\frac{\partial V^{i}}{\partial x^{j}}$. A one-form $\phi=\phi_{i}dx^{i}$
on $\mathbb{R}^{n}$ is identified as a $1\times n$ row vector function.
If $f\in C^{\infty}(\mathbb{R}^{n})$, $df$ is the one-form defined
by $df\triangleq\frac{\partial f}{\partial x^{i}}dx^{i}$. Given a
smooth function $f$ and vector field $V$, we write $Vf\triangleq df\cdot V=V^{i}\frac{\partial f}{\partial x^{i}}$.
The pairing between a one-form $\phi$ and a vector field $V$ is
obviously $\phi\cdot V$, while on the other hand we write $V\phi$
as the $1\times n$ row vector defined by $V\phi\triangleq(V\phi_{1},\cdots,V\phi_{n})$.
Note that $DV$ and $V\phi$ are local quantities that do not have
intrinsic geometric meaning.

\subsection{Pathwise differential calculus}

Let $\{X_{t}:t\geqslant0\}$ be the solution to the SDE (\ref{eq:MainRDE})
where $M=\mathbb{R}^{n}$ for now, $B_{t}$ is a $d$-dimensional
fBM with Hurst parameter $H>1/4$ and the vector fields $V_{1},\cdots,V_{d}\in C_{b}^{\infty}$.
Here $B_{t}$ is regarded as a geometric rough path and the SDE is
solved under the framework of rough path theory (cf. \cite{LQ02}).
Throughout the rest, we will assume that $B_{t}$ is realised on the
canonical path space. More specifically, the underlying probability
space is $({\cal W},{\cal B}({\cal W}),\mathbb{P})$ where ${\cal W}$
is the Banach space of $\mathbb{R}^{d}$-valued continuous paths starting
at the origin, ${\cal B}({\cal W})$ is the Borel $\sigma$-algebra
over ${\cal W}$ and $\mathbb{P}$ is the law of the fBM. The process
$B$ is taken to be the coordinate process on ${\cal W}$. Under this
set-up, solutions of differential equations driven by $B$ and stochastic
line integrals along $B$ are regarded as functionals over ${\cal W}$.

The rough path nature of (\ref{eq:MainRDE}) allows us to justify
and make use of pathwise differential calculus in the ordinary manner.
To illustrate this, we first recall the following definition of the
Cameron-Martin subspace (cf. \cite{FH14}).
\begin{defn}
The \textit{Cameron-Martin subspace} associated with fBM is the subspace
${\cal H}$ of paths $h\in{\cal W}$ that can be represented in the
form 
\[
h_{t}=\mathbb{E}[ZB_{t}],\ \ \ 0\leqslant t\leqslant T,
\]
where $Z$ is an element in the first Wiener chaos (i.e. the $L^{2}$-closure
of linear functions on ${\cal W}$ under $\mathbb{P}$). ${\cal H}$
is a Hilbert space with respect to the inner product 
\[
\langle h_{1},h_{2}\rangle_{{\cal H}}\triangleq\mathbb{E}[Z_{1}Z_{2}],
\]
where $Z_{i}$ is the chaos element associated with $h_{i}$ in its
definition ($i=1,2$).
\end{defn}

We use the following lemma to illustrate an example of the type of
pathwise calculation that will appear frequently later on. If $F:{\cal W}\rightarrow\mathbb{R}$
is a functional of $B$ and $h\in{\cal H}$ is a Cameron-Martin path,
we write 
\[
D_{h}F(w)\triangleq\frac{d}{d\varepsilon}\left|_{\varepsilon=0}\right.F(w+\varepsilon h)
\]
as the derivative of $F$ along direction $h$ at the location $w\in{\cal W}$.
The \textit{Malliavin derivative} of $F$ is the ${\cal H}$-valued
random variable defined by 
\[
DF\triangleq[h\mapsto D_{h}F]\in{\cal H}^{*}\cong{\cal H}.
\]

\begin{lem}
\label{lem:HDerSLI}Let $\phi=\phi_{i}dx^{i}$ be a $C_{p}^{\infty}$
one-form on $\mathbb{R}^{n}$. Consider the stochastic line integral
\[
F\triangleq\int_{0}^{T}\phi(dX_{t}).
\]
Then 
\begin{equation}
D_{h}F(w)=\int_{0}^{T}\big((\zeta_{T}(w)-\zeta_{t}(w))\cdot\Phi_{t}^{-1}(w)+\phi(X_{t}(w))\big)\cdot V_{\alpha}(X_{t}(w))dh_{t}^{\alpha}.\label{eq:HDerSLI}
\end{equation}
Here $\Phi_{t}(w)\triangleq\frac{\partial X_{t}(w)}{\partial x_{0}}$
denotes the Jacobian of the RDE (\ref{eq:MainRDE}) and 
\begin{equation}
\zeta_{t}(w)\triangleq\int_{0}^{t}d(\phi_{i}V_{\alpha}^{i})(X_{s}(w))\cdot\Phi_{s}(w)dw_{s}^{\alpha}.\label{eq:ZetaPath}
\end{equation}
\end{lem}

\begin{proof}
To simplify notation we will omit the dependence on $w$. By the definition
of $D_{h}F(w),$ we have
\begin{align*}
D_{h}F(w) & =\frac{d}{d\varepsilon}\left|_{\varepsilon=0}\right.\int_{0}^{T}\phi_{i}(X_{t}(w+\varepsilon h))dX_{t}^{i}(w+\varepsilon h)\\
 & =\int_{0}^{T}\frac{\partial\phi_{i}}{\partial x^{j}}(X_{t})D_{h}X_{t}^{j}dX_{t}^{i}+\int_{0}^{1}\phi_{i}(X_{t})dD_{h}X_{t}^{i}.
\end{align*}
By differentiating the SDE (\ref{eq:MainRDE}) along the direction
$h$, it is seen that $D_{h}X_{t}$ satisfies the differential equation
\begin{equation}
dD_{h}X_{t}^{i}=\frac{\partial V_{\alpha}^{i}}{\partial x^{j}}(X_{t})D_{h}X_{t}^{j}dw_{t}^{\alpha}+V_{\alpha}^{i}(X_{t})dh_{t}^{\alpha}.\label{eq:RDEDhX}
\end{equation}
As a result, we have
\begin{align}
D_{h}F(w) & =\int_{0}^{T}\frac{\partial\phi_{i}}{\partial x^{j}}(X_{t})D_{h}X_{t}^{j}V_{\alpha}^{i}(X_{t})dw_{t}^{\alpha}\nonumber \\
 & \ \ \ +\int_{0}^{T}\phi_{i}(X_{t})\big(\frac{\partial V_{\alpha}^{i}}{\partial x^{j}}(X_{t})D_{h}X_{t}^{j}dw_{t}^{\alpha}+V_{\alpha}^{i}(X_{t})dh_{t}^{\alpha}\big)\nonumber \\
 & =\int_{0}^{T}\frac{\partial}{\partial x^{j}}(\phi_{i}V_{\alpha}^{i})(X_{t})D_{h}X_{t}^{j}dw_{t}^{\alpha}+\int_{0}^{1}\phi_{i}(X_{t})V_{\alpha}^{i}(X_{t})dh_{t}^{\alpha}.\label{eq:HDerSLIPf}
\end{align}
On the other hand, the Jabocian $\Phi_{t}$ satisfies the homogeneous
linear equation 
\[
d\Phi_{t}=DV_{\alpha}(X_{t})\Phi_{t}dw_{t}^{\alpha},\ \Phi_{0}={\rm Id}.
\]
By the variational principle, it is standard that 
\begin{equation}
D_{h}X_{t}=\Phi_{t}\int_{0}^{t}\Phi_{s}^{-1}V_{\alpha}(X_{s})dh_{s}^{\alpha}.\label{eq:MDerviaJac}
\end{equation}
By using the formula (\ref{eq:MDerviaJac}) and integration by parts,
the first integral in (\ref{eq:HDerSLIPf}) can be written as
\begin{align*}
 & \int_{0}^{T}\big(d\zeta_{t}\cdot\int_{0}^{t}\Phi_{s}^{-1}V_{\alpha}(X_{s})dh_{s}^{\alpha}\big)\\
 & =\zeta_{T}\cdot\int_{0}^{T}\Phi_{s}^{-1}V_{\alpha}(X_{s})dh_{s}^{\alpha}-\int_{0}^{T}\zeta_{t}\Phi_{t}^{-1}V_{\alpha}(X_{t})dh_{t}^{\alpha}\\
 & =\int_{0}^{T}(\zeta_{T}-\zeta_{t})\Phi_{t}^{-1}V_{\alpha}(X_{t})dh_{t}^{\alpha},
\end{align*}
where $\zeta_{t}$ is the integral path defined by (\ref{eq:ZetaPath}).
The equation (\ref{eq:HDerSLI}) thus follows.
\end{proof}
\vspace{2mm}\noindent \textbf{A technical remark}. In the above proof,
we have performed pathwise integration essentially using principles
of ordinary calculus. This type of calculations can all be made rigorous
under the framework of rough path theory (cf. \cite{CHLT15,Ina14,FH14}
in which such type of calculation was used frequently and justified
carefully). For instance, the integral on the right hand side of (\ref{eq:HDerSLI})
is understood in the sense of Young, due to a variational embedding
theorem for the Cameron Martin space ${\cal H}$ proved by Friz-Victoir
(cf. \cite{FH14}). Another example is that the path $\zeta_{t}$
can be understood in the sense of RDE, namely the last component of
the triple $\Xi_{t}\triangleq(X_{t},\Phi_{t},\zeta_{t})$ which is
defined through the RDE 
\[
\begin{cases}
dX_{t}=V_{\alpha}(X_{t})dB_{t}^{\alpha},\\
d\Phi_{t}=DV_{\alpha}(X_{t})\Phi_{t}dB_{t}^{\alpha},\\
d\zeta_{t}=d(\phi\cdot V_{\alpha})(X_{t})\cdot\Phi_{t}dB_{t}^{\alpha}.
\end{cases}
\]
The assumption of $\phi\in C_{p}^{\infty}$ ensures that the stochastic
line integral $F$ has moments of all orders (indeed smooth in the
sense of Malliavin), due to standard estimates of rough integrals
and the exponential integrability of the $p$-variation of $B$ ($p>1/H$).
This is seen in exactly the same way as in \cite{Ina14}. Throughout
the rest, we will perform pathwise differential calculus of similar
kind without further justification.

\vspace{2mm} The following theorem is a standard regularity result
in the Malliavin calculus which will be used in Section \ref{sec:SLI}.
Its proof can be found in \cite{Nua06}.
\begin{thm}
\label{thm:Malliavin1D}Let $F$ be a twice differentiable random
variable on ${\cal W}$ (in the sense of Malliavin) and $F,DF,D^{2}F\in L^{p}$
for some $p>1$. Then conditional on the event $\{DF\neq0\}$, the
distribution of $F$ is absolutely continuous with respect to the
Lebesgue measure on $\mathbb{R}$.
\end{thm}

\subsection{Some terminology from differential geometry}

The notion of RDEs, stochastic line integrals as well as our non-degeneracy
criteria in Section \ref{sec:SLI} are intrinsic properties, in the
sense that they are defined in terms of the underlying vector fields
and one-forms. In particular, they are independent of the choice of
local coordinates or embedding of the state manifold into an ambient
Euclidean space. It is thus beneficial to perform some of the analysis
in geometric terms. The benefit is particularly clear in the hypoelliptic
analysis developed in Section \ref{subsec:HypoCase}. In this section,
we recall some notation from differential geometry that will be used
later on (cf. Chern-Chen-Lam \cite{CCL00}).

Let $M$ be a differentiable manifold. We denote $\Omega^{k}(M)$
($0\leqslant k\leqslant n$) as the space of (smooth) $k$-forms on
$M$. Given a (smooth) vector field $X$, the \textit{interior product}
$i(X):\Omega^{k}(M)\rightarrow\Omega^{k-1}(M)$ is defined by 
\begin{equation}
(i(X)\omega)(Y_{1},\cdots,Y_{k-1})\triangleq\omega(X,Y_{1},\cdots,Y_{k-1}).\label{eq:IntProd}
\end{equation}
The \textit{Lie derivative} $L_{X}:\Omega^{k}(M)\rightarrow\Omega^{k}(M)$
is defined by 
\begin{equation}
(L_{X}\omega)(Y_{1},\cdots,Y_{k})\triangleq X(\omega(Y_{1},\cdots,Y_{k}))-\sum_{i=1}^{k}\omega(Y_{1},\cdots,Y_{i-1},[X,Y],Y_{i+1},\cdots,Y_{k}).\label{eq:LieDer}
\end{equation}
Here $Y_{1},\cdots,Y_{k-1}$ are arbitrary vector fields, a $k$-form
is viewed as an antisymmetric $k$-linear functional on vector fields
and $Xf$ is the directional derivative of $f$ along $X$. These
two operators are related through the so-called \textit{Cartan's identity}:
\begin{equation}
d\circ i(X)+i(X)\circ d=L_{X},\label{eq:Cartan}
\end{equation}
where $d:\Omega^{k}(M)\rightarrow\Omega^{k+1}(M)$ is the exterior
derivative operator. We also recall that the \textit{exterior product}
of two one-forms $\alpha,\beta$ is defined by 
\[
\alpha\wedge\beta(X,Y)=\alpha(X)\beta(Y)-\beta(X)\alpha(Y)
\]
and the exterior derivative of one-form $\alpha$ has the following
characterisation:
\begin{equation}
d\alpha(X,Y)=X(\alpha(Y))-Y(\alpha(X))-\alpha([X,Y]),\label{eq:2Form}
\end{equation}
where $X,Y$ are arbitrary vector fields. A one-form $\alpha$ is
\textit{closed} if $d\alpha=0$. It is \textit{exact} if $\alpha=df$
for some smooth function $f.$ Every exact form is closed, and on
$\mathbb{R}^{n}$ the converse is also true. Throughout the rest,
we will simply write $\alpha\cdot X$ for the pairing $\alpha(X)$
which is also consistent with matrix notation in the Euclidean case.

It is convenient to re-interpret Lemma \ref{lem:HDerSLI} in geometric
terms. For instance, the Jacobian $\Phi_{t}:T_{x_{0}}M\rightarrow T_{X_{t}}M$
is the linear isomorphism that pushes tangent vectors at $x_{0}$
forward along the solution path by the flow of diffeomorphisms associated
with the RDE (\ref{eq:MainRDE}). The Malliavin derivative $t\mapsto D_{h}X_{t}\in T_{X_{t}}M$
is a path on the tangent bundle (cf. (\ref{eq:MDerviaJac})). The
formula (\ref{eq:HDerSLI}) can be expressed as 
\begin{equation}
D_{h}F(w)=\int_{0}^{T}\big(\int_{t}^{T}\langle\Phi_{s}^{*}d(\phi\cdot V_{\beta})(X_{s}),\Phi_{t}^{-1}V_{\alpha}(X_{t})\rangle dw_{s}^{\beta}+(\phi\cdot V_{\alpha})(X_{t})\big)dh_{t}^{\alpha},\label{eq:IntrinsicHDerSLI}
\end{equation}
where $\langle\cdot,\cdot\rangle$ denotes the pairing between cotangent
and tangent vectors at the starting point $x_{0}$. Equation (\ref{eq:IntrinsicHDerSLI})
clearly has an intrinsic meaning.

Finally, we recall an equation for the pull-back of vector fields
by the Jacobian. This equation, which played an essential role in
the proof of H\"ormander's theorem for RDEs (cf. \cite{CF10,CHLT15}),
will also be crucial for our analysis. Its proof is a straight forward
application of the chain rule.
\begin{lem}
\label{lem:Pullback}Let $W$ be a smooth vector field on $M$. Then
the path $t\mapsto\Phi_{t}^{-1}W(X_{t})\in T_{x_{0}}M$ satisfies
the equation 
\[
\Phi_{t}^{-1}W(X_{t})=W(x_{0})+\int_{0}^{t}\Phi_{s}^{-1}[V_{\alpha},W](X_{s})dw_{s}^{\alpha}.
\]
\end{lem}

\section{\label{sec:SLI}Non-degeneracy criteria for stochastic line integrals}

In this section, we establish quantitative criteria for the non-degeneracy
(i.e. existence of density) of stochastic/rough line integrals (extended
signatures) of the form

\[
\int_{0<t_{1}<\cdots<t_{m}<T}\phi_{1}(dX_{t_{1}})\cdots\phi_{m}(dX_{t_{m}}),
\]
where $X_{t}\in M$ is the solution to the RDE (\ref{eq:MainRDE})
and $\phi_{1},\cdots,\phi_{m}$ are $C_{p}^{\infty}$ one-forms. Our
results hold when $M$ is either $\mathbb{R}^{n}$ or a (compact)
manifold, but we will only work with the case when $M=\mathbb{R}^{n}$
and the vector fields $V_{\alpha}\in C_{b}^{\infty}.$ Extension of
the argument to the manifold case is routine by either working in
local charts or embedding $M$ into an ambient Euclidean space. To
illustrate the idea better, we first consider the case when $m=1$
and then extend the analysis to the case of iterated integrals.

\subsection{\label{subsec:SingleLI}Single line integrals}

We begin by considering a single line integral $F\triangleq\int_{0}^{T}\phi(dX_{t}).$
We divide the discussion into two cases: elliptic and hypoelliptic.
In the elliptic case, the result is particularly simple and neat,
while the hypoelliptic case requires a strong condition as well as
more delicate analysis.

\subsubsection{The elliptic case}

By \textit{ellipticity}, we assume that $n=d$ and the vector fields
$V_{1},\cdots,V_{d}$ in the RDE (\ref{eq:MainRDE}) linearly span
$\mathbb{R}^{d}$ at every point. In the introduction, we saw that
the line integral $F$ may fail to have a density if the one-form
$\phi$ is closed. In the elliptic case, it turns out that non-closedness
is essentially sufficient for the non-degeneracy of $F$. The main
result is stated as follows.
\begin{thm}
\label{thm:Ellip}Let $\phi$ be a $C_{p}^{\infty}$ one-form on $\mathbb{R}^{d}$.
Suppose that $d\phi\neq0$ for almost everywhere inside the support
of $\phi$. Let $E$ denote the event that ``$X_{t}$ enters the
interior of ${\rm supp}\phi$ at some time $t$''. Then conditional
on $E,$ the line integral $F$ has a density with respect to the
Lebesgue measure.
\end{thm}

Our proof of Theorem \ref{thm:Ellip}, as well as its hypoelliptic
counterpart, relies crucially on the following two properties of fBM.
Its proof is contained in \cite{CHLT15}.
\begin{lem}
\label{lem:fBMProp}(i) Let $f=(f_{1},\cdots,f_{d}):[0,T]\rightarrow\mathbb{R}^{d}$
be a deterministic path such that $\int_{0}^{T}f_{t}dh_{t}$ is well-defined
in the sense of Young for all $h\in{\cal H}$. If $\int_{0}^{T}f_{\alpha}(t)dh_{t}^{\alpha}=0$
for all $h\in{\cal H}$, then $f\equiv0$.\\
(ii) The fBM is a.s. truly rough in the sense of \cite{FH14}. As
a result, with probability one we have 
\[
\int_{0}^{t}y_{s}dB_{s}=0\ \forall t\in[0,T]\implies y\equiv0,
\]
whenever $y$ is a rough path controlled by $B$ so that the rough
integral is well-defined.
\end{lem}

\begin{rem}
In \cite{CHLT15}, these two properties are implied by a nondeterminism-type
condition which was used by the authors to establish the smoothness
of density for the RDE solution. It was proved in the same paper that
fBM satisfies their nondeterminism condition.
\end{rem}

\begin{proof}[Proof of Theorem \ref{thm:Ellip}]According to Theorem
\ref{thm:Malliavin1D}, the point is to show that $E\subseteq\{DF\neq0\}$
modulo some $\mathbb{P}$-null set $N$. First of all, let $N_{1}\subseteq{\cal W}$
be a null set such that $w$ admits a canonical rough path lifting
and is truly rough (so that Lemma \ref{lem:fBMProp} (ii) holds) for
all $w\in N_{1}^{c}.$

Now suppose that $w\in E\cap N_{1}^{c}$ is a path such that $DF(w)=0.$
According to Lemma \ref{lem:HDerSLI} and Lemma \ref{lem:fBMProp}
(i), we have 
\begin{equation}
\big((\zeta_{T}-\zeta_{t})\cdot\Phi_{t}^{-1}+\phi(X_{t})\big)\cdot V_{\alpha}(X_{t})=0\ \ \ \forall t\in[0,T],\alpha=1,\cdots d\label{eq:EllipPf1}
\end{equation}
at the driving path $w,$ where $\zeta_{t}$ is defined by (\ref{eq:ZetaPath})
and $\Phi_{t}$ is the Jacobian of the RDE. Since the vector fields
are assumed to be elliptic, the matrix $V\triangleq(V_{1},\cdots,V_{d})$
is invertible everywhere. After multiplying (\ref{eq:EllipPf1}) by
$V(X_{t})^{-1}\Phi_{t}$, we obtain that 
\[
\zeta_{T}-\zeta_{t}+\phi(X_{t})\cdot\Phi_{t}=0.
\]
Recall from the equations for $X_{t}$ and $\Phi_{t}$ that 
\begin{equation}
d\big(\phi(X_{t})\cdot\Phi_{t}\big)=\big(\frac{\partial\phi}{\partial x^{i}}V_{\alpha}^{i}\big)(X_{t})\Phi_{t}+(\phi\cdot DV_{\alpha})(X_{t})\cdot\Phi_{t}\big)dw_{t}^{\alpha}.\label{eq:EllipPf2}
\end{equation}
In view of the definition of $\zeta_{t}$ and (\ref{eq:EllipPf2}),
Lemma \ref{lem:fBMProp} (ii) implies that 
\[
\big(-d(\phi_{i}V_{\alpha}^{i})+\frac{\partial\phi}{\partial x^{i}}V_{\alpha}^{i}+\phi\cdot DV_{\alpha}\big)(X_{t})=0\ \ \ \forall t\in[0,T],\ \alpha=1,\cdots,d.
\]
By taking the $j$-th component of this equation, it is seen that
\begin{equation}
\big(\frac{\partial\phi_{j}}{\partial x^{i}}-\frac{\partial\phi_{i}}{\partial x^{j}}\big)(X_{t})V_{\alpha}^{i}(X_{t})=0\ \ \ \forall t,\alpha,j.\label{eq:EllipPf3}
\end{equation}
By ellipticity, the equation (\ref{eq:EllipPf3}) is equivalent to
the property that $(d\phi)(X_{t})=0$ for all $t$. Note that this
property holds at the particular path $w$. To summarise, we have
shown that 
\begin{equation}
w\in E\cap N_{1}^{c},\ DF(w)=0\implies(d\phi)(X_{t}(w))=0\ \forall t\in[0,T].\label{eq:EllipPf}
\end{equation}
By continuity, this is particularly true for $t\in\mathbb{Q}\cap[0,T]$.

On the other hand, by the definition of $E$ and continuity, there
is a rational time $r$ such that $X_{r}(w)\in({\rm supp}\phi)^{\circ}$.
In addition, we know from \cite{CF10} that the law of $X_{r}$ is
absolutely continuous with respect to the Lebesgue measure. Since
$\Lambda\triangleq\{x\in({\rm supp}\phi)^{\circ}:(d\phi)(x)=0\}$
is a Lebesgue null set by assumption, we have

\[
\mathbb{P}\big(X_{r}\in\Lambda\big)=0,\ \ \ \forall r\in\mathbb{Q}\cap[0,T].
\]
In view of (\ref{eq:EllipPf}), by further excluding the $\mathbb{P}$-null
set 
\[
N_{2}\triangleq\bigcup_{r\in\mathbb{Q}\cap[0,T]}\{X_{r}\in\Lambda\},
\]
we conclude that 
\[
w\in E\backslash(N_{1}\cup N_{2})\implies DF(w)\neq0.
\]
The result thus follows from Theorem \ref{thm:Malliavin1D}.

\end{proof}

Examples that satisfy the assumptions of Theorem \ref{thm:Ellip}
are generic and easy to construct. 
\begin{example}
\label{exa:EllipExam}Let $h(t)\in C_{c}^{\infty}(\mathbb{R})$ be
a function such that 
\[
h(t)>0,\ t\in(-1,1);\ h(t)=0,\ t\notin(-1,1),
\]
and $h'(t)$ is everywhere nonzero in $(-1,1)$ except at $t=0$.
Define the following one-form on $\mathbb{R}^{2}$:
\[
\phi=h(x)h(y)e^{h(y)^{2}}dx.
\]
Then $\phi$ is supported on $[-1,1]^{2}$ and 
\[
d\phi=-h(x)h'(y)(1+2h(y)^{2})e^{h(y)^{2}}dx\wedge dy.
\]
Inside its support, $d\phi=0$ precisely on the slice $y=0$ which
has zero Lebesgue measure.
\end{example}

\subsubsection{\label{subsec:HypoCase}The hypoelliptic case}

We now extend the previous analysis to the hypoelliptic case. We first
give the following key definition. Let $V_{1},\cdots,V_{d}$ be a
family of smooth vector fields on a differentiable manifold $M$.
\begin{defn}
\label{def:HorCond}We say that $V_{1},\cdots,V_{d}$ satisfy \textit{H\"ormander's condition}
if the following family of vector fields 
\[
V_{i},[V_{i},V_{j}],[V_{i},[V_{j},V_{k}]],[V_{i},[V_{j},[V_{k},V_{l}]]],\cdots\ \ \ (i,j,k,l\ \text{etc.}=1,\cdots,d)
\]
linearly span $T_{x}M$ at every $x\in M$.
\end{defn}

It is a well-known fact that under H\"ormander's condition, the solution
$X_{t}$ to the RDE (\ref{eq:MainRDE}) admits a smooth density function
with respect to the Lebesgue measure (cf. \cite{CF10,CHLT15}). In
the diffusion case, this result was first established in the seminal
work of fH\"ormander \cite{Hor67}.

We now consider a stochastic line integral $F=\int_{0}^{T}\phi(dX_{t})$,
where $X_{t}$ is the solution to the RDE (\ref{eq:MainRDE}) and
$V_{1},\cdots,V_{d}$ are $C_{b}^{\infty}$-vector fields on $M=\mathbb{R}^{n}$
that satisfy H\"ormander's condition. We shall obtain a quantitative
criterion for the non-degeneracy of $F$ and derive an explicit method
of constructing one-forms that satisfy such criterion.

\subsubsection*{A general criterion}

In order to derive a non-degeneracy criterion for $F$, as in the
elliptic case we start by assuming that $DF(w)=0$ at a given fBM
path $w$. We aim at obtaining a geometric constraint on $\phi$ which
holds at paths $w$ satisfying $DF(w)=0$ (in the elliptic case, the
constraint is closedness: $d\phi=0$). Our non-degeneracy criterion
will simply be that ``$\phi$ does not satisfy such a geometric constraint''
(in the elliptic case, $d\phi\neq0$ a.e. inside ${\rm supp}\phi$).

We begin by fixing the following notation. Given a word $I=(i_{1},\cdots,i_{k})$
over the letters $\{1,\cdots,d\}$, we set 
\[
V_{I}\triangleq[V_{i_{1}},[V_{i_{2}},\cdots,[V_{i_{k-1}},V_{i_{k}}]]].
\]
The set of finite words (respectively, of length $k$) is denoted
as ${\cal W}$ (respectively, ${\cal W}_{k}$). The following lemma
is crucial for us.
\begin{lem}
\label{lem:ConsAllDeg}We define $\{\psi_{I}:I\in{\cal W}\}$ inductively
in the following way: 
\[
\psi_{I}\triangleq0\ \ \ \text{for }I\in{\cal W}_{1},
\]
and 
\begin{equation}
\psi_{I}\triangleq d\phi(V_{i},V_{J})+V_{i}\psi_{J}\ \ \ \text{for }I=(i,J).\label{eq:PsiI}
\end{equation}
Suppose that $DF(w)=0$. Then at the path $w$, we have
\begin{equation}
(\eta\cdot\Phi^{-1}+\phi)\cdot V_{I}+\psi_{I}=0\ \ \ \forall I\in{\cal W},\label{eq:ConsAllDeg}
\end{equation}
where
\[
\eta_{t}(w)\triangleq\int_{t}^{T}d(\phi\cdot V_{\alpha})(X_{s}(w))\cdot\Phi_{s}(w)dw_{s}^{\alpha}.
\]
\end{lem}

\begin{proof}
We prove the claim by induction on the length of $I$. When $I\in{\cal W}_{1}$,
this is a restatement of (\ref{eq:EllipPf1}), whose proof clearly
does not rely on ellipticity. Suppose that (\ref{eq:ConsAllDeg})
is true for all words of length $\leqslant k$. By taking differential
with $I\in{\cal W}_{k}$, we find that
\[
d\eta\cdot(\Phi^{-1}V_{I})+\eta\cdot d(\Phi^{-1}\cdot V_{I})+V_{i}(\phi\cdot V_{I}+\psi_{I})dw_{t}^{i}=0.
\]
By the definition of $\eta$ and Lemma \ref{lem:Pullback}, we have
\[
\big(\eta\cdot\Phi^{-1}[V_{i},V_{I}]-V_{I}(\phi\cdot V_{i})+V_{i}(\phi\cdot V_{I})+V_{i}\psi_{I}\big)dw_{t}^{i}=0.
\]
It follows from Lemma \ref{lem:fBMProp} (ii) that
\[
\eta\cdot\Phi^{-1}[V_{i},V_{I}]-V_{I}(\phi\cdot V_{i})+V_{i}(\phi\cdot V_{I})+V_{i}\psi_{I}=0\ \ \ \forall i=1,\cdots,d.
\]
In addition, note that 
\begin{align*}
V_{i}(\phi\cdot V_{I})-V_{I}(\phi\cdot V_{i}) & =(V_{i}\phi)V_{I}-(V_{I}\phi)V_{i}+\phi\cdot(DV_{I}\cdot V_{i}-DV_{i}\cdot V_{I})\\
 & =d\phi(V_{i},V_{I})+\phi\cdot[V_{i},V_{I}].
\end{align*}
Therefore, we obtain that 
\[
(\eta\cdot\Phi^{-1}+\phi)[V_{i},V_{I}]+d\phi(V_{i},V_{I})+V_{i}\psi_{I}=0\ \ \ \forall i.
\]
According to the definition of $\psi_{I'}$, the above relation is
precisely the desired property for the word $I'=(i,I)$ ($i=1,\cdots,d$).
This completes the induction step.
\end{proof}
We shall make use of properly chosen local frame fields (i.e. family
of vector fields that form a basis of $T_{x}\mathbb{R}^{n}$ at every
point $x$) associated with H\"ormander's condition. As a result,
our condition on $\phi$ will be expressed locally in terms of these
frame fields. For each $x\in\mathbb{R}^{n},$ according to H\"ormander's condition
and continuity, there exists a neighbourhood $U$ of $x$ together
with subsets ${\cal I}_{1},\cdots,{\cal I}_{r}$ of words (${\cal I}_{k}\subseteq{\cal W}_{k}$),
such that 
\[
\{V_{I}:I\in{\cal I}_{1}\cup\cdots\cup{\cal I}_{r}\}
\]
form a local frame field of $\mathbb{R}^{n}$ on $U$. We may assume
that ${\rm supp}\phi$ is covered by such local ``charts''.

Now suppose that $X(w)$ passes through a local chart $U$ of ${\rm supp}\phi$
on which a local frame field 
\[
{\cal V}=\{V_{I}:I\in{\cal I}_{1}\cup\cdots\cup{\cal I}_{r}\}
\]
is chosen and fixed. Note that 
\[
|{\cal I}_{1}|+\cdots+|{\cal I}_{r}|=n.
\]
Let $W$ be the ${\rm Mat}(n,n)$-valued function on $U$ defined
by 
\[
W\triangleq(V_{I})_{I\in{\cal I}_{k},1\leqslant k\leqslant r}
\]
and set 
\[
\Theta\triangleq(\psi_{I})_{I\in{\cal I}_{k},1\leqslant k\leqslant r},
\]
where $\psi_{I}$ is defined by (\ref{eq:PsiI}). Under the assumption
$DF(w)=0,$ the relation (\ref{eq:ConsAllDeg}) can be written in
matrix form as
\[
(\eta\cdot\Phi^{-1}+\phi)\cdot W+\Theta=0
\]
Since $W$ is invertible, we have 
\begin{equation}
\eta+\phi\cdot\Phi=\Xi\cdot\Phi\label{eq:InvertedCons}
\end{equation}
where $\Xi\triangleq-\Theta\cdot W^{-1}.$ Note that (\ref{eq:InvertedCons})
holds at $w$ for all times in 
\[
L_{w}\triangleq\{t\in[0,1]:X_{t}(w)\in U\}.
\]
Let $\{\omega^{I}\}$ be the coframe field dual to ${\cal V}$. As
a one-form on $U,$ we have
\begin{equation}
\Xi=-\sum_{k=2}^{r}\sum_{I\in{\cal I}_{k}}\psi_{I}\omega^{I}\ \ \ \text{on \ensuremath{U},}\label{eq:XiCoF}
\end{equation}

\begin{lem}
\label{lem:GeoIden}Let $\omega=\omega_{i}dx^{i}$ be a one-form and
$V=V^{i}\partial_{i}$ be a vector field. Then the following two identities
hold true:

\vspace{2mm}\noindent (i) $-d(\omega\cdot V)+V\omega+\omega\cdot DV=i(V)d\omega$;\\
(ii) $V\omega+\omega\cdot DV=L_{V}\omega.$
\end{lem}

\begin{proof}
By the definition (\ref{eq:LieDer}) of the Lie derivative, we have
\begin{align*}
(L_{V}\omega)(\partial_{i}) & =V\omega_{i}-\omega\cdot[V,\partial_{i}]=V^{j}\frac{\partial\omega_{i}}{\partial x^{j}}+\omega_{j}\frac{\partial V^{j}}{\partial x^{i}}\\
 & =(V\omega+\omega\cdot DV)\cdot\partial_{i}.
\end{align*}
This justifies the relation in (ii). The relation of (i) is a simple
consequence of (ii) and Cartan's identity (\ref{eq:Cartan}).
\end{proof}
\begin{lem}
\label{lem:HypoCrit}Suppose that $DF(w)=0$ and $X(w)$ passes through
the local chart $U$. Then at the path $w$, we have 
\begin{equation}
i(V_{\alpha})d\phi=L_{V_{\alpha}}\Xi\label{eq:HypoConsGen}
\end{equation}
for all $\alpha=1,\cdots,d$ and $t\in L_{w}$.
\end{lem}

\begin{proof}
By taking differential of the relation (\ref{eq:InvertedCons}), we
obtain that
\[
\big(-d(\phi\cdot V_{\alpha})\cdot\Phi+(V_{\alpha}\cdot\phi)\cdot\Phi+\phi\cdot DV_{\alpha}\cdot\Phi\big)dw^{\alpha}=\big((V_{\alpha}\Xi)\cdot\Phi+\Xi\cdot DV_{\alpha}\cdot\Phi\big)dw^{\alpha}.
\]
After cancelling $\Phi$ on both sides, Lemma \ref{lem:fBMProp} (ii)
implies that 
\[
-d(\phi\cdot V_{\alpha})+V_{\alpha}\phi+\phi\cdot DV_{\alpha}=V_{\alpha}\Xi+\Xi\cdot DV_{\alpha}\ \ \ \forall\alpha=1,\cdots,d\ \text{and }t\in L_{w}.
\]
The result follows immediately from Lemma \ref{lem:GeoIden}.
\end{proof}
Equivalently, Lemma \ref{lem:HypoCrit} suggests that if the relation
(\ref{eq:HypoConsGen}) does not hold on $U$ and if $X(w)$ passes
through $U$, then $DF(w)\neq0$. As a consequence, along the same
lines of argument as in the elliptic case, we have proved the following
result, which is the main theorem in this section giving a quantitative
non-degeneracy criterion for the line integral $F$.
\begin{thm}
\label{thm:GenMain}Let $\phi$ be a $C_{p}^{\infty}$ one-form on
$\mathbb{R}^{n}$. Suppose that the support of $\phi$ is covered
by local charts $U$ on which suitable local frame fields ${\cal V}$
are chosen and fixed. For each $U$, define the local one-form $\Xi_{U}$
on $U$ by (\ref{eq:XiCoF}) with respect to the coframe dual to ${\cal V}.$
Suppose that on each chart $U,$ we have
\begin{equation}
i(V_{\alpha})d\phi-L_{V_{\alpha}}\Xi_{U}\neq0\ \ \ \text{a.e. on}\ U\cap{\rm supp\phi}\label{eq:GenMain}
\end{equation}
for some $\alpha=1,\cdots,d.$ Then conditional on the event that
``$X$ enters the support of $\phi$'', the stochastic line integral
$\int_{0}^{T}\phi(dX_{t})$ has a density with respect to the Lebesgue
measure.
\end{thm}

\begin{rem}
The condition (\ref{eq:GenMain}) is stronger than the non-closedness
condition $d\phi\neq0$ a.e. obtained in the elliptic case. Indeed,
it is obvious that 
\[
d\phi=0\implies\Xi_{U}=0\implies i(V_{\alpha})d\phi-L_{V_{\alpha}}\Xi_{U}=0.
\]
\end{rem}

\subsubsection*{An explicit method of construction}

The next basic question is whether there are rich examples of one-forms
that satisfy the non-degeneracy criteria derived in the previous sections.
In the elliptic case, the non-closedness condition is fairly easy
to achieve. In the hypoelliptic case, there is also a rich class of
one-forms (at least as generic as pairs of smooth functions) that
satisfy the condition (\ref{eq:GenMain}). In what follows, we discuss
a general and explicit method of constructing them.

We first recall some basic notation from sub-Riemannian geometry.
Suppose that $\{V_{1},\cdots,V_{d}\}$ are given smooth vector fields
on a differentiable manifold $M$ which satisfy H\"ormander's condition.
Define ${\cal D}_{1}$ to be the $C^{\infty}(M)$-module generated
by $\{V_{1},\cdots,V_{d}\}$. Equivalently, for each $x\in M$, ${\cal D}_{1}(x)$
is the subspace of $T_{x}(M)$ defined by 
\[
{\cal D}_{1}(x)={\rm Span}\{V_{1}(x),\cdots,V_{d}(x)\},\ \ \ x\in M.
\]
Inductively, define 
\[
{\cal D}_{k}\triangleq{\cal D}_{k-1}+[{\cal D}_{1},{\cal D}_{k-1}],\ \ \ k\geqslant2,
\]
where $[{\cal D}_{1},{\cal D}_{k-1}]$ denote the $C^{\infty}(M)$-module
generated by $\{[X,Y]:X\in{\cal D}_{1},Y\in{\cal D}_{k-1}\}$. Elements
in ${\cal D}_{k}$ are linear combinations of $\{V_{I}:|I|\leqslant k\}$
with smooth coefficients. According to H\"ormander's condition, at
every $x\in M$ there is a smallest integer $r(x)$ such that ${\cal D}_{r(x)}(x)=T_{x}M$.
Observe that 
\[
\{0\}=:{\cal D}_{0}(x)\subseteq{\cal D}_{1}(x)\subseteq{\cal D}_{2}(x)\subseteq\cdots\subseteq{\cal D}_{r(x)}(x).
\]
The list of integers 
\[
\dim{\cal D}_{1}(x)<\dim{\cal D}_{2}(x)<\cdots<\dim{\cal D}_{r(x)}(x)
\]
is known as the \textit{growth vector} of $\{V_{1},\cdots,V_{d}\}$
at $x$. A point $x$ is a \textit{regular point }if the growth vector
is constant near $x$. The set of regular points is open and dense
in $M$.

The following simple algebraic lemma allows us to choose preferable
local frame fields to work with.
\begin{lem}
\label{lem:GoodChart}Let $x_{0}\in M$ be a regular point. There
exists a neighbourhood $U$ of $x_{0}$ and a collection of words
$({\cal I}_{1},\cdots,{\cal I}_{r})$ ($r\triangleq r(x_{0})$, ${\cal I}_{k}\subseteq{\cal W}_{k}$),
such that the following two properties hold true for each $k=1,\cdots,r$:

\vspace{2mm}\noindent (i) ${\cal I}_{k}\subseteq{\cal I}_{1}\times{\cal I}_{k-1}$;\\
(ii) $\{V_{I}:I\in{\cal I}_{1}\cup\cdots\cup{\cal I}_{k}\}$ is a
local frame field of ${\cal D}_{k}$ on $U$.
\end{lem}

\begin{proof}
We construct ${\cal I}_{k}$ by induction. First of all, let $U_{0}$
be a neighbourhood of $x_{0}$ on which the growth vector is constant.
Choose ${\cal I}_{1}$ so that $\{V_{i}(x_{0}):i\in{\cal I}_{1}\}$
form a basis of ${\cal D}_{1}(x_{0}).$ By continuity, there exists
$U_{1}\subseteq U_{0}$ such that $\{V_{i}(x):i\in{\cal I}_{1}\}$
are linearly independent for each $x\in U_{1}.$ Since $\dim{\cal D}_{1}$
is constant on $U_{1}$, we see that $\{V_{i}:i\in{\cal I}_{1}\}$
is a local frame field of ${\cal D}_{1}$ on $U_{1}.$

Now suppose that a neighbourhood $U_{k}$ and ${\cal I}^{(k)}={\cal I}_{1}\cup\cdots\cup{\cal I}_{k}$
have been obtained to satisfy the required properties. We claim that
\begin{equation}
{\cal D}_{k+1}(x_{0})={\rm Span}\big\{ V_{I}(x_{0}),[V_{i},V_{J}](x_{0}):i\in{\cal I}_{1},I\in{\cal I}^{(k)},J\in{\cal I}_{k}\big\}.\label{eq:GoodNbh}
\end{equation}
Indeed, let $W\in{\cal D}_{1}(U_{k})$ and $Z\in{\cal D}_{k}(U_{k})$.
By the induction hypothesis, we can write 
\[
W=\sum_{i\in{\cal I}_{1}}f_{i}V_{i},\ Z=\sum_{J\in{\cal I}_{1}\cup\cdots\cup{\cal I}_{k}}g_{J}V_{J}
\]
where $f_{i},g_{I}\in C^{\infty}(U_{k})$. It follows that 
\[
[W,Z]=\sum_{i\in{\cal I}_{1}}\sum_{J\in{\cal I}_{1}\cup\cdots\cup{\cal I}_{k}}\big((f_{i}V_{i}g_{J})V_{J}-(g_{J}V_{J}f_{i})V_{i}+f_{i}g_{J}[V_{i},V_{J}]\big).
\]
For $J\in{\cal I}^{(k-1)},$ since $[V_{i},V_{J}]\in{\cal D}_{k}$
is a $C^{\infty}(M)$-linear combination of $V_{I}$ ($I\in{\cal I}^{(k)}$)
on $U_{k}$, the claim (\ref{eq:GoodNbh}) follows immediately. Note
that $\{V_{I}(x_{0})\}:I\in{\cal I}^{(k)}\}$ are already linearly
independent. As a result, we can choose a collection ${\cal I}_{k+1}$
of $(i,J)$ with $i\in{\cal I}_{1},J\in{\cal I}_{k}$ such that 
\[
\{V_{I}(x_{0}),[V_{i},V_{J}](x_{0}):I\in{\cal I}^{(k)},(i,J)\in{\cal I}_{k+1}\}
\]
form a basis of ${\cal D}_{k+1}(x_{0}).$ By continuity and the constant
dimensionality of ${\cal D}_{k+1}$ on $U_{k}$, we see that $\{V_{I}:I\in{\cal I}^{(k+1)}\}$
is a local frame field of ${\cal D}_{k+1}$ on some $U_{k+1}\subseteq U_{k}.$
From the construction, it is also clear that ${\cal I}_{k+1}\subseteq{\cal I}_{1}\times{\cal I}_{k}$.
\end{proof}
\begin{rem}
We know from Property (i) that ${\cal I}_{k}\neq\emptyset$ for all
$k$.
\end{rem}

We now derive a general method of constructing one-forms that satisfy
Theorem \ref{thm:GenMain}. Let $(U;{\cal V}=\{V_{I}:I\in{\cal I}_{1}\cup\cdots\cup{\cal I}_{r}\})$
be a chosen local frame field that satisfies the properties in Lemma
\ref{lem:GoodChart}. For $I=(i,J),$ we denote $\llangle d\phi,V_{I}\rrangle\triangleq d\phi(V_{i},V_{J}).$
From the definition of $\psi_{I}$ (cf. (\ref{eq:PsiI})) and Property
(i) of Lemma \ref{lem:GoodChart}, it is not hard to see that
\[
\llangle d\phi,V_{I}\rrangle=0\ \forall I\in{\cal I}_{2}\cup\cdots\cup{\cal I}_{r}\implies\Xi=0,
\]
where $\Xi$ is the one-form defined by (\ref{eq:XiCoF}) with respect
to the local frame field ${\cal V}$. In addition, since $\Xi\cdot V_{\alpha}=0$
for all $\alpha\in{\cal I}_{1}$, we have 
\begin{equation}
i(V_{\alpha})d\phi-L_{V_{\alpha}}\Xi=i(V_{\alpha})d(\phi-\Xi).\label{eq:IntLie}
\end{equation}
As a result, a sufficient condition for (\ref{eq:GenMain}) to hold
on $U$ is that:

\vspace{2mm}\noindent (A) $\llangle d\phi,V_{I}\rrangle=0$ for all
$I\in{\cal I}_{2}\cup\cdots\cup{\cal I}_{r};$\\
(B) $i(V_{\alpha})d\phi\neq0$ a.e. on $U$ for some $\alpha\in{\cal I}_{1}$.

\vspace{2mm} We shall reduce the above two conditions to a more explicit
set of relations in terms of coefficients. To this end, let $\{\omega^{I}:I\in{\cal I}^{(r)}\triangleq{\cal I}_{1}\cup\cdots\cup{\cal I}_{r}\}$
be the coframe dual to ${\cal V}$ and express $\phi$ on $U$ as
\[
\phi=\sum_{I\in{\cal I}^{(r)}}c_{I}\omega^{I},
\]
where $c_{I}\in C^{\infty}(U)$. Let us fix a total ordering $\prec$
on ${\cal I}^{(r)}$ such that $I\prec J$ if $|I|<|J|$. For $I,J,K\in{\cal I}^{(r)},$
we set
\begin{align*}
\Lambda_{JK}^{I} & \triangleq d\omega^{I}(V_{J},V_{K})=V_{J}(\omega^{I}(V_{K}))-V_{K}(\omega^{I}(V_{J}))-\omega^{I}([V_{J},V_{K}])=-\omega^{I}([V_{J},V_{K}]).
\end{align*}
It follows that 
\begin{align*}
d\phi & =\sum_{I}\big(dc_{I}\wedge\omega^{I}+\sum_{J\prec K}c_{I}\Lambda_{JK}^{I}d\omega^{J}\wedge d\omega^{K}\big)\\
 & =\sum_{I}\big(\sum_{J}V_{J}c_{I}\omega^{J}\wedge\omega^{I}-\sum_{J\prec K}c_{I}\omega^{I}([V_{J},V_{K}])d\omega^{J}\wedge d\omega^{K}\big)\\
 & =\sum_{I\prec J}(V_{I}c_{J}-V_{J}c_{I}-\sum_{K}c_{K}\omega^{K}([V_{I},V_{J}]))\omega^{I}\wedge\omega^{J}.
\end{align*}

Let $c_{i}$ ($i\in{\cal I}_{1}$) be an arbitrary family of smooth
functions on $U$. Given $I=(i,j)\in{\cal I}_{2},$ since $i,j\in{\cal I}_{1}$,
we have
\[
\llangle d\phi,V_{I}\rrangle=\pm(V_{i}c_{j}-V_{j}c_{i}-\sum_{K}c_{K}\omega^{K}(V_{I})))=\pm(V_{i}c_{j}-V_{j}c_{i}-c_{I}\big).
\]
As a result, by setting
\[
c_{I}\triangleq V_{i}c_{j}-V_{j}c_{i},\ \ \ I=(i,j)\in{\cal I}_{2},
\]
we conclude that $\llangle d\phi,V_{I}\rrangle=0$ for all $I\in{\cal I}_{2}$.
Inductively on $k$, for $I=(i,J)\in{\cal I}_{k}$ we set 
\[
c_{I}\triangleq V_{i}c_{J}-V_{J}c_{i}
\]
where $c_{J}$ has already been defined since $J\in{\cal I}_{k-1}.$
It then follows that 
\[
\llangle d\phi,V_{I}\rrangle=0\ \ \ \forall I\in{\cal I}_{2}\cup\cdots\cup{\cal I}_{r}
\]
on $U$. In particular, the aforementioned Condition (A) holds. For
Condition (B), note that 
\begin{equation}
i(V_{\alpha})d\phi=\sum_{J:J\neq\alpha}(V_{\alpha}c_{J}-V_{J}c_{\alpha}-\sum_{K}c_{K}\omega^{K}([V_{\alpha},V_{J}]))\omega^{J}\label{eq:Innerdphi}
\end{equation}
for each $\alpha\in{\cal I}_{1}$. As a result, Condition (B) boils
down to requiring that at least one of the $\omega^{J}$-coefficients
in (\ref{eq:Innerdphi}) is a.e. nonzero on $U$.

To summarise, we have obtained the following result which provides
an explicit method of constructing one-forms that satisfy the criterion
(\ref{eq:GenMain}).
\begin{thm}
\label{thm:NonDegHypo}Let $c_{i}\in C_{c}^{\infty}(U)$ ($i\in{\cal I}_{1}$)
be given and define $c_{I}$ ($I\in{\cal I}_{k}$) inductively by
\[
c_{I}\triangleq V_{i}c_{J}-V_{J}c_{i},\ \ \ I=(i,J)\in{\cal I}_{k}.
\]
Suppose that for some $\alpha\in{\cal I}_{1}$ and $J\in{\cal I}_{r}$,
we have
\begin{equation}
V_{\alpha}c_{J}-V_{J}c_{\alpha}-\sum_{K}c_{K}\omega^{K}([V_{\alpha},V_{J}])\neq0\ \ \ \text{a.e. on }U.\label{eq:ExpCond}
\end{equation}
Then conditional on the event that ``$X$ enters the support of $\phi$'',
the stochastic line integral $\int_{0}^{T}\phi(dX_{t})$ has a density
with respect to the Lebesgue measure.
\end{thm}

\begin{rem}
The left hand side of (\ref{eq:ExpCond}) is an expression involving
up to the $r$-th derivatives of $c_{i}$ ($i\in{\cal I}_{1}$). The
property (\ref{eq:ExpCond}) is essentially generic for functions
$c_{i}\in C_{c}^{\infty}(U)$ ($i\in{\cal I}_{1}$).
\end{rem}

\subsubsection*{The step-two case and the Heisenberg group}

Let us consider the simplest hypoelliptic situation, i.e. when $d=2,$
$\dim M=3$ and the vector fields 
\[
{\cal V}=\{V_{1},V_{2},V_{3}\triangleq[V_{1},V_{2}]\}
\]
form a basis of $T_{x}M$ at every point $x\in M$(i.e. a global frame
field over $M$). In this case, Theorems \ref{thm:GenMain} and \ref{thm:NonDegHypo}
are simplified substantially. Let $\{\omega^{1},\omega^{2},\omega^{3}\}$
be the coframe of ${\cal V}.$ The definition (\ref{eq:XiCoF}) of
the one-form $\Xi$ reads
\[
\Xi=-d\phi(V_{1},V_{2})\omega^{3}.
\]
According to the identity (\ref{eq:IntLie}) and the anti-symmetry
of $d(\phi-\Xi)$ as a bilinear form on vector fields, the condition
(\ref{eq:GenMain}) in Theorem \ref{thm:GenMain} is equivalent to
that 
\[
d(\phi+d\phi(V_{1},V_{2})\omega^{3})\neq0\ \ \ \text{a.e.}\ \text{on supp}\phi.
\]
In addition, Conditions (A) and (B) in the last section simply reads
\[
d\phi(V_{1},V_{2})=0\ \text{and }d\phi\neq0\ \text{a.e.}\ \text{on supp\ensuremath{\phi}.}
\]
In terms of coefficients of $\phi$ with respect to $\{\omega^{1},\omega^{2},\omega^{3}\}$,
we have the following direct corollary of Theorem \ref{thm:NonDegHypo}.
\begin{cor}
\label{cor:Step2Explicit}Consider a one-form
\begin{equation}
\phi=c_{1}\omega^{1}+c_{2}\omega^{2}+(V_{1}c_{2}-V_{2}c_{1})\omega^{3},\label{eq:S2Form}
\end{equation}
where $c_{1},c_{2}\in C_{p}^{\infty}(M)$ . Suppose that $d\phi\neq0$
a.e. inside the support of $\phi$. Then conditional on the event
that ``$X$ enters ${\rm supp\phi}$'', the stochastic line integral
$\int_{0}^{T}\phi(dX_{t})$ has a density with respect to the Lebesgue
measure.
\end{cor}

We conclude with an explicit example: the \textit{Heisenberg group}.
More precisely, we consider $M=\mathbb{R}^{3}$, where the vector
fields $V_{1},V_{2}$ are given by

\[
V_{1}=\partial_{x}-y\partial_{z},\ V_{2}=\partial_{y}+x\partial_{z}
\]
respectively. In this case, the solution to the RDE (\ref{eq:MainRDE})
is explicitly given by the original fBM $B$ coupled with the associated
\textit{L\'evy area process} 
\[
X_{t}=\big(B_{t}^{x},B_{t}^{y},\int_{0}^{t}B_{s}^{x}dB_{s}^{y}-\int_{0}^{t}B_{s}^{y}dB_{s}^{x}\big)_{0\leqslant t\leqslant T}.
\]
By explicit calculation, it is easily seen that $[V_{1},V_{2}]=2\partial_{z}$.
In particular, ${\cal V}\triangleq\{V_{1},V_{2},[V_{1},V_{2}]\}$
is a global frame field. Its coframe is found to be 
\[
\omega^{1}=dx,\ \omega^{2}=dy,\ \omega^{3}=\frac{y}{2}dx-\frac{x}{2}dy+\frac{1}{2}dz.
\]
Let $\phi=c_{i}\omega^{i}$ where $c_{i}\in C_{p}^{\infty}(\mathbb{R}^{3})$.
Under Cartesian coordinates, we have
\begin{equation}
\phi=\big(c_{1}+\frac{yc_{3}}{2}\big)dx+\big(c_{2}-\frac{xc_{3}}{2}\big)dy+\frac{1}{2}c_{3}dz.\label{eq:PhiHeisenberg}
\end{equation}
Let us further assume that $c_{1},c_{2}$ depend only on the $x,y$
coordinates. Define 
\[
c_{3}\triangleq-V_{2}c_{1}+V_{1}c_{1}=-\partial_{y}c_{1}+\partial_{x}c_{2},
\]
so that $d\phi(V_{1},V_{2})=0$ as seen before. Note that $c_{3}$
also depends only on $x,y$. We obtain from (\ref{eq:PhiHeisenberg})
that 
\begin{equation}
\begin{cases}
d\phi(\partial_{x},\partial_{z})=\frac{1}{2}\partial_{x}c_{3}=\frac{1}{2}\big(-\partial_{xy}^{2}c_{1}+\partial_{xx}^{2}c_{2}\big),\\
d\phi(\partial_{y},\partial_{z})=\frac{1}{2}\partial_{y}c_{3}=\frac{1}{2}\big(-\partial_{yy}^{2}c_{1}+\partial_{xy}^{2}c_{2}\big).
\end{cases}\label{eq:Heisenberg}
\end{equation}

As a consequence, as long as the functions $(c_{1},c_{2})$ are chosen
such that 
\begin{equation}
\big(-\partial_{xy}^{2}c_{1}+\partial_{xx}^{2}c_{2}\big)\cdot\big(-\partial_{yy}^{2}c_{1}+\partial_{xy}^{2}c_{2}\big)\neq0\ \ \ \text{a.e. in {\rm supp\ensuremath{\phi}}},\label{eq:HeisenbergCond}
\end{equation}
the non-degeneracy of the line integral $\int_{0}^{T}\phi(dX_{t})$
holds. Since there are no a priori constraints on $c_{1},c_{2},$
the property (\ref{eq:HeisenbergCond}) is apparently generic. 

\subsection{\label{subsec:ILI}Iterated line integrals}

We now turn to the stuy of an extended signature 
\[
F=\int_{0<t_{1}<\cdots<t_{m}<T}\phi_{1}(dX_{t_{1}})\cdots\phi_{m}(dX_{t_{m}})\ \ \ (m\geqslant2).
\]
We consider two typical situations: (i) the supports of the one-forms
$\phi_{1},\cdots,\phi_{m}$ are mutually disjoint, or (ii) they all
have common support. As we will see, in the first case the conditions
provided by Theorem \ref{thm:GenMain} (imposed on each $\phi_{i}$)
continue to ensure the non-degeneracy of $F$. In the second case,
we demonstrate that it is possible to have all $\phi_{i}$'s being
exact while $F$ is non-degenerate, which is surprising in contrast
to the case of $m=1$.

We first prepare a lemma that will be used in both cases. It is a
natural extension of (\ref{eq:EllipPf1}).
\begin{lem}
For $k=1,\cdots,m,$ we set 
\begin{align}
G_{t}^{k} & \triangleq\int_{0<t_{1}<\cdots<t_{k-1}<t}\phi_{1}(dX_{t_{1}})\cdots\phi_{k-1}(dX_{t_{k-1}}),\label{eq:Gk}\\
H_{t}^{k} & \triangleq\int_{t<t_{k+1}<\cdots<t_{m}<T}\phi_{k+1}(dX_{t_{k+1}})\cdots\phi_{m}(dX_{t_{m}}),\label{eq:Hk}
\end{align}
where $G_{t}^{1}=H_{t}^{m}\triangleq1.$ Suppose that $DF(w)=0$.
Then at the path $w$ we have 
\begin{equation}
\sum_{k=1}^{m}\big(\int_{t}^{T}G_{s}^{k}H_{s}^{k}d\zeta_{s}^{k}\cdot\Phi_{t}^{-1}+G_{t}^{k}H_{t}^{k}\phi_{k}(X_{t})\big)\cdot V_{\alpha}(X_{t})=0\label{eq:LengthOneIIG}
\end{equation}
for all $\alpha=1,\cdots,d$ and $t\in[0,T]$, where 
\begin{equation}
\zeta_{t}^{k}\triangleq\int_{0}^{t}d(\phi_{k}\cdot V_{\alpha})(X_{s})\Phi_{s}dw_{s}^{\alpha}.\label{eq:Zetak}
\end{equation}
\end{lem}

\begin{proof}
As in the proof of Lemma \ref{lem:HDerSLI}, given any $h\in{\cal H}$
we have
\begin{align*}
D_{h}F(w) & =\sum_{k=1}^{m}\int_{0<t_{1}<\cdots<t_{k}<\cdots<t_{m}<T}\phi_{1}(dX_{t_{1}})\cdots D_{h}\phi_{k}(dX_{t_{k}})\cdots\phi_{m}(dX_{t_{m}})\\
 & =\sum_{k=1}^{m}\big(\int_{0<\cdots<t_{k}<\cdots<T}\cdots(d\zeta_{t_{k}}^{k}\cdot\eta_{t_{k}})\cdots\\
 & \ \ \ +\int_{0<\cdots<t_{k}<\cdots<T}\cdots(\phi_{k}\cdot V_{\alpha})(X_{t_{k}})dh_{t_{k}}^{\alpha}\cdots\big)\\
 & =\sum_{k=1}^{m}\big(\int_{0}^{T}G_{t}^{k}H_{t}^{k}(d\zeta_{t}^{k}\cdot\eta_{t})+\int_{0}^{T}G_{t}^{k}H_{t}^{k}(\phi_{k}\cdot V_{\alpha})(X_{t})dh_{t}^{\alpha}\big)\\
 & =:A_{1}+A_{2},
\end{align*}
where $\eta_{t}\triangleq\int_{0}^{t}\Phi_{s}^{-1}V_{\alpha}(X_{s})dh_{s}^{\alpha}.$
The same integration by parts argument as in the proof of Lemma \ref{lem:HDerSLI}
yields that
\[
A_{1}=\sum_{k=1}^{m}\int_{0}^{T}\int_{t}^{T}G_{s}^{k}H_{s}^{k}d\zeta_{s}^{k}\cdot\Phi_{t}^{-1}V_{\alpha}(X_{t})dh_{t}^{\alpha}.
\]
As a consequence, we have 
\[
D_{h}F(w)=\sum_{k=1}^{m}\int_{0}^{T}\big(\int_{t}^{T}G_{s}^{k}H_{s}^{k}d\zeta_{s}^{k}\cdot\Phi_{t}^{-1}V_{\alpha}(X_{t})+G_{t}^{k}H_{t}^{k}(\phi_{k}\cdot V_{\alpha})(X_{t})\big)dh_{t}^{\alpha}.
\]
Since $D_{h}F(w)=0$ for all $h\in{\cal H},$ the result thus follows
from Lemma \ref{lem:fBMProp} (i).
\end{proof}

\subsubsection{The case of disjoint supports}

Let $\phi_{1},\cdots,\phi_{m}$ be smooth one-forms such that ${\rm supp}\phi_{i}\cap{\rm supp}\phi_{j}=\emptyset$
for all $i\neq j$. Define $E$ to be the event that ``there exist
times $t_{1}<\cdots<t_{m}$ such that $X_{t_{i}}\in({\rm supp}\phi_{i})^{\circ}$
for all $i$''. From the definition of $F,$ it is not hard to see
in a deterministic way that the line integral $F$ is identically
zero on $E^{c}$. Our main result in this case is stated as follows.
\begin{thm}
\label{thm:ExtSigDisSup}Suppose that each $\phi_{i}$ satisfies the
conditions in Theorem \ref{thm:GenMain}. Then conditional on the
event $E$, the extended signature $F$ has a density with respect
to the Lebesgue measure.
\end{thm}

The following lemma, which is an extension of Lemma \ref{lem:ConsAllDeg},
is needed for our proof of Theorem \ref{thm:ExtSigDisSup}.
\begin{lem}
For each $k=1,\cdots,m$, we define $\{\psi_{k,I}:I\in{\cal W}\}$
by $\psi_{k,I}\triangleq0$ if $I\in{\cal W}_{1}$ and 
\[
\psi_{k,I}\triangleq d\phi_{k}(V_{i},V_{J})+V_{i}\psi_{k,J}
\]
for $I=(i,J)$. Suppose that $DF(w)=0$. Then at the path $w$, we
have
\begin{equation}
(\rho_{t}\cdot\Phi_{t}^{-1}+\sum_{k=1}^{m}G_{t}^{k}H_{t}^{k}\phi_{k})\cdot V_{I}+\sum_{k=1}^{m}G_{t}^{k}H_{t}^{k}\psi_{k,I}=0\ \ \ \forall I\in{\cal W},t\in[0,T],\label{eq:ConsAllDegIIG}
\end{equation}
where $\rho_{t}\triangleq\sum_{k=1}^{m}\int_{t}^{T}G_{s}^{k}H_{s}^{k}d\zeta_{s}^{k}$
and $G_{t}^{k},H_{t}^{k},\zeta_{t}^{k}$ are defined by (\ref{eq:Gk},
\ref{eq:Hk}, \ref{eq:Zetak}) respectively.
\end{lem}

\begin{proof}
We prove the claim by induction on the length of the word $I$. When
$I\in{\cal W}_{1}$, the claim reduces to the equation (\ref{eq:LengthOneIIG}).
Suppose that (\ref{eq:ConsAllDegIIG}) is true for all words of length
$\leqslant k$. By differentiating (\ref{eq:ConsAllDegIIG}) with
$I\in{\cal W}_{k}$, we have 
\begin{align*}
d\rho\cdot(\Phi^{-1}V_{I})+\rho\cdot d(\Phi^{-1}V_{I})+\sum_{k}d(G_{t}^{k}H_{t}^{k})\cdot(\phi_{k}\cdot V_{I}+\psi_{k,I})\\
+\sum_{k}G_{t}^{k}H_{t}^{k}V_{i}(\phi_{k}\cdot V_{I}+\psi_{k,I})dw^{i} & =0.
\end{align*}
Recall that 
\[
d\rho_{t}=-\sum_{k}G_{t}^{k}H_{t}^{k}d(\phi_{k}\cdot V_{i})\cdot\Phi_{t}dw_{t}^{i},\ d(\Phi_{t}^{-1}V_{I})=\Phi_{t}^{-1}\cdot[V_{i},V_{I}]dw_{t}^{i}.
\]
As a result, we have
\begin{align}
\rho\cdot\Phi^{-1}[V_{i},V_{I}] & +\sum_{k}G_{t}^{k}H_{t}^{k}\big(-V_{I}(\phi_{k}\cdot V_{i})+V_{i}(\phi_{k}\cdot V_{I})+V_{i}\psi_{k,I}\big)\nonumber \\
+\sum_{k}G_{t}^{k}H_{t}^{k+1} & \phi_{k}\wedge\phi_{k+1}(V_{i},V_{I})+\sum_{k}G_{t}^{k}H_{t}^{k+1}(\psi_{k+1,I}\phi_{k}-\psi_{k,I}\phi_{k+1})\cdot V_{i}=0\label{eq:ConstAllDegIIGPf}
\end{align}
for all $i.$ Since ${\rm supp}\phi_{k}\cap{\rm supp}\phi_{k+1}=\emptyset$,
it is readily seen that
\[
\phi_{k}\wedge\phi_{k+1}=0,\ \psi_{k+1,I}\phi_{k}-\psi_{k,I}\phi_{k+1}=0.
\]
In addition, note that 
\[
V_{i}(\phi_{k}\cdot V_{I})-V_{I}(\phi_{k}\cdot V_{i})=d\phi_{k}(V_{i},V_{I})+\phi_{k}\cdot[V_{i},V_{I}].
\]
The equation (\ref{eq:ConstAllDegIIGPf}) thus reduces to
\[
\big(\rho\cdot\Phi^{-1}+\sum_{k}G_{t}^{k}H_{t}^{k}\phi_{k}\big)\cdot[V_{i},V_{I}]+\sum_{k}G_{t}^{k}H_{t}^{k}\big(d\phi_{k}(V_{i},V_{I})+V_{i}(\psi_{k,I})\big)=0.
\]
By the definition of $\{\psi_{k,I}:I\in{\cal W}\}$, the last expression
is equivalent to that 
\[
\big(\rho\cdot\Phi^{-1}+\sum_{k}G_{t}^{k}H_{t}^{k}\phi_{k}\big)\cdot V_{I'}+\sum_{k}G_{t}^{k}H_{t}^{k}\psi_{k,I'}=0
\]
where $I'=(i,I)$. Since $I\in{\cal W}_{k}$ and $i\in\{1,\cdots,d\}$
are arbitrary, we conclude that (\ref{eq:ConsAllDegIIG}) is true
for words of length $k+1$.
\end{proof}
We now prove Theorem \ref{thm:ExtSigDisSup} by induction on the degree
of $F$.

\begin{proof}[Proof of Theorem \ref{thm:ExtSigDisSup}] Consider the
following slightly more general claim:

\vspace{2mm}\noindent (${\bf P}_{m}$) Let $\phi_{1},\cdots,\phi_{m}$
be smooth one-forms with disjoint support and each of them satisfies
the conditions in Theorem \ref{thm:GenMain}. For each pair of $s<t\in[0,T]$,
let $E_{s,t}$ be the event that ``$X$ visits $({\rm supp}\phi_{1})^{\circ},\cdots,({\rm supp}\phi_{m})^{\circ}$
in order over $[s,t]$''. Then 
\[
\left.\int_{s<t_{1}<\cdots<t_{m}<t}\phi_{1}(dX_{t_{1}})\cdots\phi_{m}(dX_{t_{m}})\right|_{E_{s,t}}
\]
admits a density with respect the Lebesgue measure.

\vspace{2mm}\noindent We are going to prove (${\bf P}_{m}$) by induction
on $m$. The case when $m=1$ is just Theorem \ref{thm:GenMain}.
Suppose that the claim is true for iterated integrals of degree less
than $m$ and consider an $m$-th order integral
\[
F=\int_{s<t_{1}<\cdots<t_{m}<t}\phi_{1}(dX_{t_{1}})\cdots\phi_{m}(dX_{t_{m}}).
\]
We wish to show that 
\begin{equation}
w\in E_{s,t}\cap N^{c}\implies DF(w)\neq0,\label{eq:EventRelIIGPf}
\end{equation}
where $N$ is a suitable $\mathbb{P}$-null set to be excluded.

Suppose that $w\in E_{s,t}$ and $DF(w)=0$. Let $k$ be fixed and
consider a time $u$ such that $X_{u}\in({\rm supp}\phi_{k})^{\circ}$
and $X|_{[s,u]}$ (respectively, $X|_{[u,t]}$) visits the supports
of $\phi_{1},\cdots,\phi_{k-1}$ (respectively, of $\phi_{k+1},\cdots,\phi_{m}$).
Such a time $u$ exists as $w\in E_{s,t}$. By the assumption of the
theorem, we may take a chart $U$ near $X_{u}$ on which a local frame
field $\{V_{I}:I\in{\cal I}_{1}\cup\cdots\cup{\cal I}_{r}\}$ is defined
and
\begin{equation}
i(V_{\alpha})d\phi_{k}-L_{V_{\alpha}}\Xi_{k}\neq0\ \ \ \text{a.e. on}\ U\label{eq:NonDegConsIIGPf}
\end{equation}
for some $\alpha$, where under the notation of Section \ref{subsec:HypoCase}
we set 
\[
\Xi_{k}\triangleq\Theta_{k}W^{-1},\ \Theta_{k}\triangleq(\psi_{k,I})_{I\in{\cal I}_{l},1\leqslant l\leqslant r},W\triangleq(V_{I})_{I\in{\cal I}_{l},1\leqslant l\leqslant r}\ \ \ \text{on \ensuremath{U}}.
\]
In a small time neighbourhood $v\in(u-\varepsilon,u+\varepsilon)$,
the equation (\ref{eq:ConsAllDegIIG}) yields 
\[
(\rho\cdot\Phi^{-1}+G_{v}^{k}H_{v}^{k}\phi_{k})\cdot W+G_{v}^{k}H_{v}^{k}\cdot\Theta_{k}=0\iff\rho+G_{v}^{k}H_{v}^{k}\phi_{k}\cdot\Phi=G_{v}^{k}H_{v}^{k}\Xi_{k}\cdot\Phi.
\]
Note that the above relation holds at $k$ (not summing over $k$!)
near $X_{u}.$ By differentiating both sides with respect to $w_{t}^{\alpha},$
we obtain that 
\begin{align}
 & G_{v}^{k}H_{v}^{k}\big(-d(\phi_{k}\cdot V_{\alpha})+V_{\alpha}\phi_{k}+\phi_{k}\cdot DV_{\alpha}\big)+d(G_{v}^{k}H_{v}^{k})\phi_{k}\nonumber \\
 & =G_{v}^{k}H_{v}^{k}\cdot\big(V_{\alpha}\Xi_{k}+\Xi_{k}\cdot DV_{\alpha}\big)+d(G_{v}^{k}H_{v}^{k})\Xi_{k}\label{eq:IIGDSPf}
\end{align}
for all $\alpha$ and $v\in(u-\varepsilon,u+\varepsilon).$

Next, we observe that 
\[
d(G_{v}^{k}H_{v}^{k})=G_{v}^{k-1}H_{v}^{k}\phi_{k-1}(dX_{v})-G_{v}^{k}H_{v}^{k+1}\phi_{k+1}(dX_{v})=0,
\]
since $X_{v}\in{\rm supp}\phi_{k}$ for $v$ close to $u$. As a result,
the equation (\ref{eq:IIGDSPf}) reduces to
\[
G_{v}^{k}H_{v}^{k}\big(-d(\phi_{k}\cdot V_{\alpha})+V_{\alpha}\phi_{k}+\phi_{k}\cdot DV_{\alpha}\big)=G_{v}^{k}H_{v}^{k}\cdot\big(V_{\alpha}\Xi_{k}+\Xi_{k}\cdot DV_{\alpha}\big).
\]
By using the relations
\begin{align*}
-d(\phi_{k}\cdot V_{\alpha})+V_{\alpha}\phi_{k}+\phi_{k}\cdot DV_{\alpha} & =i(V_{\alpha})d\phi_{k},\\
V_{\alpha}\Xi_{k}+\Xi_{k}\cdot DV_{\alpha} & =L_{V_{\alpha}}\Xi_{k},
\end{align*}
we obtain that
\begin{equation}
G_{v}^{k}H_{v}^{k}\big(i(V_{\alpha})d\phi_{k}-L_{V_{\alpha}}\Xi_{k}\big)=0\label{eq:IterPf}
\end{equation}
for all $\alpha$ and $v\in(u-\varepsilon,u+\varepsilon)$. According
to the assumption (\ref{eq:NonDegConsIIGPf}), we conclude that either
$X_{v}$ lives on some Lebesgue null set $C\subseteq U$, or $G_{v}^{k}H_{v}^{k}=0.$

For each $v,$ we set 
\begin{align*}
E_{v}' & \triangleq\{\exists t_{1}<\cdots<t_{k-1}\in(0,v):X_{t_{i}}\in({\rm supp}\phi_{i})^{\circ}\},\\
E_{v}'' & \triangleq\{\exists t_{k+1}<\cdots<t_{m}\in(0,v):X_{t_{i}}\in({\rm supp}\phi_{i})^{\circ}\}
\end{align*}
respectively. To summarise, by continuity we have obtained from (\ref{eq:IterPf})
that 
\begin{align*}
 & w\in E_{s,t}\cap\{DF=0\}\\
 & \implies w\in N\triangleq\bigcup_{r\in\mathbb{Q}\cap(s,t)}\big(\{X_{r}\in C\}\cup(\{G_{r}^{k}=0\}\cap E_{r}')\cup(\{H_{r}^{k}=0\}\cap E_{r}'')\big).
\end{align*}
Since $X_{r}$ has a density, we know that $\{X_{r}\in C\}$ is a
$\mathbb{P}$-null set. In addition, since $G_{r}^{k}$ and $H_{r}^{k}$
are iterated line integrals with degree less than $m$, by the induction
hypothesis both of $G_{r}^{k}|_{E_{r}'}$ and $H_{r}^{k}|_{E_{r}''}$
have densities. In particular,
\[
\{G_{r}^{k}=0\}\cap E_{r}')\cup(\{H_{r}^{k}=0\}\cap E_{r}'')
\]
is also a $\mathbb{P}$-null set. As a result, $\mathbb{P}(N)=0$
and the desired relation (\ref{eq:EventRelIIGPf}) follows. In other
words, we conclude that $DF\neq0$ a.s. on $E_{s,t}$, which implies
the existence of conditional density by Theorem \ref{thm:Malliavin1D}.
This completes the induction step for the claim (${\bf P}_{m}$).

\end{proof}

\subsubsection{The case of common support}

Next, we assume that the supports of $\phi_{1},\cdots,\phi_{m}$ have
a common intersection $S$. Our aim here is to demonstrate a surprising
fact that the extended signature $F$ can still be non-degenerate
even when all the $\phi_{i}$'s are exact and compactly supported
(i.e. $\phi_{i}=df_{i}$ where $f_{i}\in C_{c}^{\infty}(S)$). As
we mentioned in the introduction, this is not possible when $m=1$
(cf. Remark \ref{rem:Deg1Exact} as well). Our result in this case
is stated as follows. We only consider the elliptic situation.
\begin{prop}
Consider an elliptic RDE (\ref{eq:MainRDE}) where $X_{0}=x_{0}\in\mathbb{R}^{d}$.
Let $f_{1},\cdots,f_{m}$ be compactly supported smooth functions.
Suppose that the two-forms 
\[
df_{1}\wedge df_{2},\cdots,df_{m-1}\wedge df_{m}
\]
are linearly independent at $x_{0}$. Then the extended signature
\[
F\triangleq\int_{0<t_{1}<\cdots<t_{m}<T}(df_{1})(dX_{t_{1}})\cdots(df_{m})(dX_{t_{m}})
\]
has a density with respect to the Lebesgue measure.
\end{prop}

\begin{proof}
Write $\phi_{k}\triangleq df_{k}.$ Let $w$ be an fBM path such that
$DF(w)=0$. According to the equation (\ref{eq:LengthOneIIG}) and
ellipticity, we have 
\[
\sum_{k}\int_{t}^{1}G_{s}^{k}H_{s}^{k}d\zeta_{s}^{k}+\sum_{k}G_{t}^{k}H_{t}^{k}\phi_{k}\cdot\Phi_{t}=0\ \ \ \forall t\in[0,T].
\]
By taking differentiation with respect to $w_{t}^{\alpha}$, we find
that 
\begin{align*}
\sum_{k}G_{t}^{k}H_{t}^{k}\big(-d(\phi_{k}\cdot V_{\alpha})+V_{\alpha}(\phi_{k}) & +\phi_{k}\cdot DV_{\alpha}\big)\\
+\sum_{k}G_{t}^{k}H_{t}^{k+1}\big((\phi_{k}\cdot V_{\alpha})\phi_{k+1} & -(\phi_{k+1}\cdot V_{\alpha})\phi_{k}=0,
\end{align*}
which is equivalent to that 
\[
i(V_{\alpha})\sum_{k}\big(G_{t}^{k}H_{t}^{k}d\phi_{k}+G_{t}^{k}H_{t}^{k+1}\phi_{k}\wedge\phi_{k+1}\big)=0
\]
for all $\alpha=1,\cdots,d$ and $t\in[0,T].$ Again by ellipticity
and the fact that $d\phi_{k}=d^{2}f_{k}=0$, we have 
\begin{equation}
\sum_{k}G_{t}^{k}H_{t}^{k+1}\phi_{k}\wedge\phi_{k+1}=0\label{eq:ExactCons}
\end{equation}
for all $t\in[0,T]$ at the path $w$.

We first consider the case when $m=2$. In this case, the relation
(\ref{eq:ExactCons}) simply reads
\[
(\phi_{1}\wedge\phi_{2})(X_{t}(w))=0\ \ \ \forall t\in[0,T].
\]
By taking $t=0$, we reach a contradiction as $\phi_{1}\wedge\phi_{2}(x_{0})\neq0$
by the assumption. Next, we consider the case when $m=3.$ In this
case, the relation (\ref{eq:ExactCons}) becomes 
\[
H_{t}^{2}\phi_{1}\wedge\phi_{2}+G_{t}^{2}\phi_{2}\wedge\phi_{3}=0.
\]
By the linear independence assumption and continuity, when $t$ is
small we have $H_{t}^{2}=G_{t}^{2}=0.$ In particular, 
\begin{equation}
G_{t}^{2}=\int_{0}^{t}\phi_{1}(dX_{s})=f_{1}(X_{t})-f_{1}(x_{0})=0\ \ \ \forall t\ \text{small.}\label{eq:ExactSubM}
\end{equation}
On the other hand, since $df_{1}(x_{0})\neq0$ (otherwise the linear
independence assumption cannot hold), there exists a neighbourhood
$U$ of $x_{0}$ such that 
\[
P\triangleq\{x\in U:f_{1}(x)=f_{1}(x_{0})\}
\]
is an $(n-1)$-dimensional submanifold in $U$. In particular, the
event 
\[
N\triangleq\bigcup_{r\in\mathbb{Q}_{+}}\{X_{r}\in P\}
\]
is a $\mathbb{P}$-null set. Note that the property (\ref{eq:ExactSubM})
implies that $N$ happens. Consequently, in both cases $m=2,3$, we
see that $DF(w)\neq0$ a.s. The existence of density thus follows.

Now suppose that the claim is true for iterated integrals of degree
$m-2$ where $m\geqslant4.$ . For the degree $m$ case, by taking
$k=m-1$ in (\ref{eq:ExactCons}) we have 
\[
DF(w)=0\implies G_{t}^{m-1}=\int_{0<t_{1}<\cdots<t_{m-2}<t}\phi_{1}(dX_{s})\cdots\phi_{m-2}(dX_{s})=0
\]
when $t$ is small. In particular, 
\[
\{DF=0\}\subseteq\bigcup_{r\in\mathbb{Q}_{+}}\{G_{r}^{m-1}=0\},
\]
which is a $\mathbb{P}$-null set since $G_{r}^{m-1}$ has a density
by the induction hypothesis. Therefore, $DF\neq0$ a.s. and the claim
holds for the degree-$m$ case. The result thus follows by induction.
\end{proof}
\begin{rem}
\label{rem:Deg1Exact}In contrast, when $m=1,$ the stochastic line
integral of a compactly supported exact form will never have a density.
Indeed, let $f$ be a compactly supported smooth function. Then
\[
F\triangleq\int_{0}^{T}(df)(dX_{t})=f(X_{T})-f(x_{0}).
\]
According to \cite[Theorem 1.5]{GOT21}, the density of $X_{T}$ is
everywhere strictly positive. It follows that 
\[
\mathbb{P}(X_{T}\in({\rm supp}f)^{c})>0.
\]
In particular, there is a positive probability that $F=-f(x_{0})$.
As a result, $F$ cannot have a density. Nonetheless, if we allow
${\rm supp}f=\mathbb{R}^{n}$ it is clearly possible that $F$ has
a density. For instance, take $f(x)=|x|^{2}$ with $X_{t}$ being
a Brownian motion.
\end{rem}

\section{\label{sec:UoS}An application: signature uniqueness for RDEs}

In this section, we discuss an application of Theorem \ref{thm:NonDegHypo}
to the probabilistic signature uniqueness problem. We first give the
definition of the signature transform of a rough path (cf. \cite{LCL07}).
Let $T((\mathbb{R}^{n}))\triangleq\stackrel[m=0]{\infty}{\prod}(\mathbb{R}^{n})^{\otimes m}$
denote the algebra of formal tensor series over $\mathbb{R}^{n}$
where $(\mathbb{R}^{n})^{\otimes0}\triangleq\mathbb{R}$.
\begin{defn}
Let ${\bf X}=({\bf X}_{t})_{0\leqslant t\leqslant T}$ be a rough
path over $\mathbb{R}^{n}$. The \textit{signature} of $\mathbf{X}$
is the formal tensor series defined by 
\begin{equation}
S({\bf X})\triangleq\big(1,\int_{0}^{T}d{\bf X}_{t},\cdots,\int_{0<t_{1}<\cdots<t_{m}<T}d{\bf X}_{t_{1}}\otimes\cdots\otimes d{\bf X}_{t_{m}},\cdots\big)\in T((\mathbb{R}^{n})).\label{eq:SigDef}
\end{equation}
\end{defn}

\begin{rem}
If ${\bf X}$ is a continuous path in $\mathbb{R}^{n}$ with bounded
variation, the iterated integrals in (\ref{eq:SigDef}) are all defined
in the classical sense of Lebesgue-Stieltjes. In the rough path case,
the well-definedness of $S({\bf X})$ follows from a basic extension
theorem of Lyons (cf. \cite{LQ02}).
\end{rem}

After extracting coordinates, the signature $S({\bf X})$ consists
of a countable family of numbers associated with the path ${\bf X}$.
It can be viewed as the pathwise / deterministic analogue of moments
of a random variable. There are two basic reasons of considering the
signature transform:

\vspace{2mm}\noindent (i) {[}\textit{The signature uniqueness theorem}{]}
Every (geometric) rough path is uniquely determined by its signature
up to tree-like pieces (cf. \cite{HL10,BGLY16}). Here a tree-like
piece is a portion along which the path travels out and reverses back
to cancel itself. \\
(ii) The signature $S({\bf X})$ has nice algebraic and analytic properties
that are concealed at the level of paths (cf. \cite{LCL07,Reu93}).

\vspace{2mm} In the probabilistic setting, the signature uniqueness
theorem may take a stronger form as we do not expect tree-like pieces
to appear for a suitably non-degenerate stochastic process. Below
is the main result in this section which extends earlier probabilistic
works \cite{LQ12,GQ16,BG15}. To reduce technicalities, we only consider
the elliptic or step-two hypoelliptic case.
\begin{thm}
\label{thm:ASSigUniq}Consider an $n$-dimensional RDE (\ref{eq:MainRDE})
driven by a $d$-dimensional fractional Brownian motion. Suppose that
the vector fields $\{V_{1},\cdots,V_{d}\}$ are $C_{b}^{\infty}$
and we are in one of the following two situations:

\vspace{2mm}\noindent (i) $n=d$ and the vector fields are elliptic;\\
(ii) $n=3,$ $d=2$ and the vector fields satisfy H\"ormander's condition.

\vspace{2mm}\noindent Then with probability one, every sample path
of the solution process $X=\{X_{t}:0\leqslant t\leqslant T\}$ is
uniquely determined by its signature up to reparametrisation.
\end{thm}

\begin{rem}
We expect the result to be true for the general hypoelliptic case
of arbitrary order, although the construction of relevant one-forms
(cf. Condition (ND) below) may be technically more involved in the
general case.
\end{rem}

Such a probabilistic uniqueness theorem was first established by Le
Jan and Qian \cite{LQ12} for the Brownian motion case. The result
was later extended to the cases of hypoelliptic diffusions in \cite{GQ16}
and Gaussian processes in \cite{BG15}. These works were largely based
on a technique developed in \cite{LQ12}, which was formalised in
\cite{BG15} down to the verification of three key conditions in the
context of a general stochastic process $X$. The first two conditions
are: (i) $X$ can be lifted as a rough path in a canonical way and
(ii) $X_{t}$ has a density for each $t>0$. These two conditions
are naturally satisfied for hypoelliptic RDEs. The last condition
is stated as follows.

\vspace{2mm}\noindent \textbf{Non-degeneracy Condition} (ND). For
any cube $H$ in $\mathbb{R}^{n}$, there exists a smooth one-form
$\phi$ supported in $H,$ such that conditional on the event that
``$X$ enters $H$ at some time'', the stochastic line integral
$\int_{0}^{T}\phi(dX_{t})$ is a.s. non-zero.

\vspace{2mm} It was proved in \cite{BG15} that the above three conditions
imply the signature uniqueness theorem for a general stochastic process
$X.$ As a result, in order to prove the aforementioned Theorem \ref{thm:ASSigUniq},
it remains to verify Condition (ND) under the given assumptions. Before
doing so, for the sake of completeness, we briefly recapture the main
strategy of \cite{LQ12} and explain at a conceptual level how Condition
(ND) leads to the signature uniqueness property.

\vspace{2mm} \noindent \textit{Step one}. Decompose the state space
$\mathbb{R}^{n}$ into disjoint cubes of order $\varepsilon$ with
narrow gaps $\delta$ ($\delta<<\varepsilon$). Label the cubes by
a set $L$ ($L=\mathbb{Z}^{n}$ in \cite{LQ12}).\\
\textit{Step two}. For each cube $H_{z}$ ($z\in L$), construct a
one-form $\phi_{z}$ supported in $H_{z}$ according to Condition
(ND). For each word $w=(z_{1},\cdots,z_{m})$ over $L$, one can define
the associated extended signature 
\[
[\phi_{z_{1}},\cdots,\phi_{z_{m}}]_{0,T}\triangleq\int_{0<t_{1}<\cdots<t_{m}<T}\phi_{z_{1}}(dX_{t_{1}})\cdots\phi_{z_{m}}(dX_{t_{m}})
\]
along the path $X.$ As a consequence of an algebraic property of
the signature, these extended signatures are all uniquely determined
by the signature of $X$. \\
\textit{Step three}. Due to Condition (ND), there exists a unique
word $w$ of maximal length, with respect to which the extended signature
is non-zero. This word precisely corresponds to the discrete route
of the path $X$ in the given space discretisation. As a result, the
signature of $X$ uniquely determines its discrete route.\\
\textit{Step four}. As we refine the discretisation (i.e. sending
$\varepsilon,\delta\rightarrow0$), the discrete route converges to
the original sample path $X$ in a suitable sense. Therefore, the
signature uniquely determines the trajectory of $X$.

\vspace{2mm} The rest of this section is devoted to the proof of
Theorem \ref{thm:ASSigUniq}.

\subsection*{Proof of Theorem \ref{thm:ASSigUniq}: Verification of Condition
(ND)}

In the elliptic case, we can use Example \ref{exa:EllipExam} to explicitly
construct one-forms satisfying Condition (ND). According to Theorem
\ref{thm:Ellip}, conditional on $X$ entering the cube $H$, the
associated line integral $\int_{0}^{T}\phi(dX_{t})$ (for $\phi$
given by Example \ref{exa:EllipExam}) has a density. This clearly
implies that its value is a.s. non-zero.

We now consider the step-two hypoelliptic case. Suppose that $n=3,$
$d=2$ and ${\cal V}=\{V_{1},V_{2},[V_{1},V_{2}]\}$ form a global
frame field of $\mathbb{R}^{3}$. We use the method of Corollary \ref{cor:Step2Explicit}
to construct suitable one-forms. Recall from (\ref{eq:S2Form}) that
such one-forms are given by 
\[
\phi=c_{1}\omega^{1}+c_{2}\omega^{2}+(V_{1}c_{2}-V_{2}c_{1})\omega^{3},
\]
where $\{\omega^{i}\}$ is the coframe of ${\cal V}$ and $c_{1},c_{2}$
are arbitrary smooth functions supported in the cube $H$. We want
to choose $\phi$ with ${\rm supp}\phi=H$ and $d\phi\neq0$ a.e.
in $H$. Note that $d\phi(V_{1},V_{2})=0.$ Hence we have to look
at $d\phi(V_{i},[V_{1},V_{2}]).$ Straightforward calculation yields
\[
d\phi(V_{i},[V_{1},V_{2}])=V_{i}(V_{1}c_{2}-V_{2}c_{1})-[V_{1},V_{2}]c_{i}-\langle\phi,[V_{i},[V_{1},V_{2}]]\rangle,\ \ \ i=1,2.
\]
We will set $c_{2}=0$, so that 
\begin{equation}
d\phi(V_{2},[V_{1},V_{2}])=-V_{2}(V_{2}c_{1})-\langle\omega^{1},[V_{2},[V_{1},V_{2}]]\rangle\cdot c_{1}+\langle\omega^{3},[V_{2},[V_{1},V_{2}]]\rangle\cdot V_{2}c_{1}.\label{eq:UniqPf}
\end{equation}
In other words, we want to construct $c_{1}$ with ${\rm supp}c_{1}=H$,
such that the above expression is a.e. non-zero in $H$.

According to \cite[Chap. 1, Theorem 4.3]{CCL00}, a non-degenerate
vector field locally generates coordinate curves. Since we will eventually
refine the space discretisation, we may assume without loss of generality
that $H$ is contained in a coordinate chart $[U;x,y,z]$ of $\mathbb{R}^{3}$
where $V_{2}=\partial_{x}.$ To simplify notation, we further assume
that $H$ is the unit cube
\[
H=\{(x,y,z):\max\{|x|,|y|,|z|\}<1\}
\]
under the above coordinate system. We define
\[
c_{1}(x,y,z;\lambda)\triangleq h_{\lambda}(x)\eta(y,z),
\]
where $\lambda>0$ is a parameter to be chosen later on, 
\[
h_{\lambda}(x)\triangleq\begin{cases}
e^{-\frac{\lambda}{1-x^{2}}}, & |x|<1;\\
0, & |x|\geqslant1,
\end{cases}
\]
and $\eta(y,z)$ is a given smooth function supported on $\bar{H}_{y,z}\triangleq\{(y,z):\max\{|y|,|z|\}\leqslant1\}$
which is strictly positive in the interior. Under such choice of $c_{1},$
the equation (\ref{eq:UniqPf}) can be concisely written as 
\[
-d\phi(V_{2},[V_{1},V_{2}])=\big(h''_{\lambda}(x)+f(x,y,z)h_{\lambda}'(x)+g(x,y,z)h_{\lambda}(x)\big)\eta(y,z),
\]
where $f,g$ are known $C^{\infty}$-functions. Our proof will be
concluded from the following lemma.
\begin{lem}
\label{lem:Sard}There exists $\lambda>0,$ such that
\[
N_{\lambda}\triangleq\{(x,y,z)\in H:h''_{\lambda}(x)+f(x,y,z)h_{\lambda}'(x)+g(x,y,z)h_{\lambda}(x)=0\}
\]
is a Lebesgue null set.
\end{lem}

\begin{rem}
It will be clear from the proof below that Lemma \ref{lem:Sard} holds
for all $\lambda$ outside a suitable null set of $(0,\infty)$. For
our purpose, we only need one such $\lambda$.
\end{rem}

\begin{proof}
Explicit calculation shows that 
\[
h''_{\lambda}(x)+f(x,y,z)h_{\lambda}'(x)+g(x,y,z)h_{\lambda}(x)=\frac{h_{\lambda}(x)}{(1-x^{2})^{4}}\cdot\Phi_{\lambda}(x,y,z),
\]
where 
\begin{equation}
\Phi_{\lambda}(x,y,z)=4x^{2}\lambda^{2}-2(1-x^{2})(1+3x^{2}+x(1-x^{2})f)\lambda+(1-x^{2})^{4}g.\label{eq:Quadratic}
\end{equation}
Observe that $\Phi_{\lambda}(x,y,z)$ is a quadratic polynomial in
$\lambda.$ It is easy to see that $(x,y,z)\in N_{\lambda}\cap\{x\neq0\}$
if and only if 
\[
\lambda=\frac{-p\pm\sqrt{\Delta}}{8x^{2}}\ \text{and }\Delta\geqslant0,
\]
where $p,\Delta$ are known $C^{\infty}$-functions on $H$ that can
be expressed explicitly in terms of $f,g$ ($\Delta$ is the discriminant
of (\ref{eq:Quadratic})).

We now consider the following three smooth functions:
\[
\psi_{\pm}\triangleq\frac{-p\pm\sqrt{\Delta}}{8x^{2}},\ q\triangleq\frac{-p}{8x^{2}},
\]
where $\psi_{\pm}$ are defined on $E\triangleq\{\Delta>0\}\cap\{x\neq0\}$
(could possibly be empty) and $q$ is defined on $H\cap\{x\neq0\}$.
Recall that the critical set of a smooth function $F:U\rightarrow\mathbb{R}$
consists of those points in $U$ at which $\nabla F=0$. The classical
Sard's theorem (cf. \cite[Chap. 2]{Mil97}) asserts that the image
of the critical set of a smooth function is a Lebesgue null set. Let
$Y_{\pm},Z$ be the critical sets of $\psi_{\pm},q$ respectively.
It follows that 
\[
C\triangleq\psi_{+}(Y_{+})\cup\psi_{-}(Y_{-})\cup q(Z)
\]
is a Lebesgue null set in $\mathbb{R}.$ As a result, there exists
at least one $\lambda\in(0,\infty)\cap C^{c}$. We fix one such $\lambda.$
Then each of $\psi_{+}^{-1}(\lambda),$ $\psi_{-}^{-1}(\lambda)$,
$q^{-1}(\lambda)$ is either empty or a two-dimensional sub-manifold
in $H$. The result thus follows from the observation that 
\[
N_{\lambda}\cap\{x\neq0\}\subseteq\psi_{+}^{-1}(\lambda)\cup\psi_{-}^{-1}(\lambda)\cup q^{-1}(\lambda).
\]
Note that the slice $\{x=0\}$ is a Lebesgue null set and has no effect
on our discussion.
\end{proof}
If we choose $\lambda$ as in Lemma \ref{lem:Sard}, for the resulting
one-form $\phi$ we have 
\[
d\phi(V_{2},[V_{1},V_{2}])\neq0
\]
except on a low dimensional manifold which has zero Lebesgue measure.
Therefore, $d\phi\neq0$ a.e. inside the support of $\phi$. The Condition
(ND) is then a consequence of Corollary \ref{cor:Step2Explicit}.

\end{document}